\documentclass[a4paper,reqno,10pt]{amsart}

\pdfoutput=1
\usepackage[USenglish]{babel}
\usepackage[utf8x]{inputenc}

\usepackage[style=american]{csquotes}
\usepackage{ucs}
\usepackage{yfonts}
\usepackage{amsmath,amsthm,amscd,amssymb}
\usepackage{MnSymbol}
\usepackage[all]{xy}
\usepackage{smartref}
\usepackage{algorithmic}
\usepackage{tikz}
\usepackage{hyperref}
\usepackage{algorithm, caption}
\usepackage[lite,shortalphabetic]{amsrefs}
\usepackage{faktor}
\usepackage{marvosym}
\usepackage{caption}

\theoremstyle{definition}
\newtheorem{defn}{Definition}[section]
\newtheorem{ex}[defn]{Example}
\newtheorem{notation}[defn]{Notation}
\newtheorem{remark}[defn]{Remark}

\newtheorem*{acknowledgement}{Acknowledgement}

\newtheorem*{remark*}{Remark}

\theoremstyle{plain}
\newtheorem{theorem}[defn]{Theorem}
\newtheorem{corollary}[defn]{Corollary}
\newtheorem{lemma}[defn]{Lemma}
\newtheorem{prop}[defn]{Proposition}

\newcommand{\abs}[1]{\left\lvert #1 \right\rvert} %print |..|
\newcommand{\gnrt}[1]{\left\langle #1 \right\rangle} %print <...>

\newcommand{\Z}{\mathbb{Z}}

\newcommand{\R}{\mathbb{R}}

\renewcommand{\P}{\mathbb{P}}

\newcommand{\relint}{\textnormal{relint}\,}

\newcommand{\Star}{\textnormal{Star}}

\newcommand{\mn}{\mathcal{M}_{0,n}}
\newcommand{\mk}[1]{\mathcal{M}_{0,#1}^\textnormal{trop}}
\newcommand{\val}{\textnormal{val}}
\newcommand{\ft}{\textnormal{ft}}
\newcommand{\curly}[1]{\mathcal{#1}}

\newcommand{\polymake}{\texttt{polymake}}
\newcommand{\atint}{\texttt{a-tint}}

\newcommand{\tsr}{\otimes} %tensor
\newcommand{\wo}{\setminus} %print a backslah

\definecolor{comment}{rgb}{.2,.2,.2}

\newcommand{\hw}[1]{\mathbb{H}_{#1}^{\textnormal{trop}}}
\newcommand{\hwt}[1]{\tilde{\mathbb{H}}_{#1}^{\textnormal{trop}}}

\newcommand{\ev}{\textnormal{ev}}

\numberwithin{equation}{section}
\definecolor{DarkGreen}{rgb}{0,0.5,0}
\definecolor{DarkRed}{rgb}{0.8,0,0}

\newcommand{\atcommand}{atint  $>$ }

\renewcommand{\relint}{\textnormal{relint}}

\newtheorem*{theorem*}{Theorem}
\newtheorem*{corollary*}{Corollary}

\definecolor{darkgreen}{rgb}{0,0.5,0}

\tikzset{
  >=latex  
}

\begin{document}
\title{Combinatorics of tropical Hurwitz cycles}
\author {Simon Hampe}
\address {Simon Hampe, Technische Universität Berlin, Institut für Mathematik, Sekretariat MA 6-2, Straße des 17. Juni 136, 10623 Berlin, Germany}
\email {hampe@math.tu-berlin.de}
\begin{abstract}
  We study properties of the tropical double Hurwitz loci defined by Bertram, Cavalieri and Markwig. We show that all such loci are connected in codimension one. If we mark preimages of simple ramification points, then for a generic choice of such points the resulting cycles are weakly irreducible, i.e.\ an integer multiple of an irreducible cycle. We study how Hurwitz cycles can be written as divisors of rational functions and show that they are numerically equivalent to a tropical version of a representation as a sum of boundary divisors. The results and counterexamples in this paper were obtained with the help of \atint, an extension for \polymake\ for tropical intersection theory.
\end{abstract}
\maketitle
% \tableofcontents

\section{Introduction}

Roughly speaking, Hurwitz numbers count covers of $\mathbb{P}^1$ by complex curves $C$ of some genus $g$ --- but with a given degree and some special ramification profile over a certain number of points\footnote{In fact, one can consider this problem in even greater generality by counting covers $C \to C'$, where $C$ and $C'$ are curves of prescribed genera $g$ and $g'$.}. For example, \emph{single} Hurwitz numbers require the cover $C \to \mathbb{P}^1$ to have a specific ramification profile over some special point (usually $\infty$) and only simple ramification elsewhere. These numbers have played a significant role in the study of the intersection theory of the moduli spaces $\overline{\curly{M}}_{g,n}$ of curves. The \emph{ELSV formula} \cite{elsvformula} relates Hurwitz numbers to certain intersection products of tautological classes on $\overline{\curly{M}}_{g,n}$. This was then used by Okounkov and Pandharipande to prove Witten's conjecture \cite{ophurwitz} - though the first proof of this is of course due to Kontsevich \cite{kintersectiontheory}.

To obtain \emph{double} Hurwitz numbers, we fix the ramification over two points in $\mathbb{P}^1$, usually 0 and $\infty$. These numbers not only occur in algebraic geometry, but also in representation theory and combinatorics - thus providing a strong connection between a wide variety of disciplines. An overview over the different definitions of double Hurwitz numbers can for example be found in \cite{jhurwitzoverview}. An ELSV-type formula has been conjectured by Goulden, Jackson and Vakil in \cite{gjvtowardsgeometry}, where it is also shown that these numbers are piecewise polynomial in terms of the ramification profile. By convention, one writes the profile as $x \in \Z^n$ with $\sum x_i = 0$. The interpretation of this is that the positive part $x^+$ gives the ramification profile over $0$ and the negative part $x^-$ gives the ramification profile over $\infty$. A special feature of double Hurwitz numbers is the fact that the number of simple ramification points only depends on the length of the ramification profile, not on the multiplicities. The number of additional simple ramification points is then $n-2+2g$. This fact will be very helpful in defining higher-dimensional cycles.

The generalization to Hurwitz cycles is achieved by letting one or more of the images of simple ramification points \enquote{move around} in $\mathbb{P}^1$. In the general case, these loci were defined and studied by Graber and Vakil in \cite{gvrelativevirtual}. In the genus 0 case, Bertram, Cavalieri and Markwig proved that these cycles are linear combinations of cycles with coefficients that are piecewise polynomial in the entries of the ramification profile \cite{bcmhurwitz}. They also considered tropical versions $\hw{k}(x,p)$ and $\hwt{k}(x,p)$, respectively, of double Hurwitz loci and showed that their combinatorics relate very nicely to the combinatorics of the different strata of the algebraic loci via dualizing of graphs. Here $\hwt{k}(x,p)$ differs from $\hw{k}(x,p)$ in that the preimages of the simple ramification points $p_i$ are also marked.

Higher-dimensional Hurwitz loci were a key ingredient in the study of tautological classes of $\overline{M}_{g,n}$ in \cite{gvrelativevirtual}. For tropical geometers, they are also of particular interest in the search for a more conceptual approach to enumerative geometry. So far, tropical enumerative results could only be translated to results in algebraic geometry by using correspondence theorems (e.g.\ \cites{menumerative, cjmhurwitz,bbmopenhurwitz, nstoricdegenerations}). These theorems only apply to very specific enumerative problems. A more general result which could, for example, relate intersection rings of algebraic and tropical moduli spaces, would make tropical enumerative geometry much more powerful. The fact that Hurwitz numbers (and possibly, Hurwitz cycles) are so closely related to intersection theory on $\overline{M}_{g,n}$ makes them a good starting point for this approach. A natural question to ask in this context is whether the algebraic Hurwitz cycle somehow tropicalizes onto the tropical one. In \cite{bcmhurwitz}, the tropical Hurwitz cycles are obtained by translating a Gromov-Witten type formula to its tropical analogue. While the definition is rather simple and involves only the well-known tropical moduli space of rational curves, the cycles itself are rather large (in terms of ambient dimension and number of polyhedral cells) even for small examples and difficult to study \enquote{by hand}. This makes it very hard to prove a more concrete tropicalization result. We will therefore start by studying the tropical Hurwitz cycles and their properties to make them more accessible.

There are two main properties we want to consider in this paper: connectedness in codimension one and irreducibility. The first is relevant for computational purposes, as well as a necessary condition for the second property. Irreducibility itself is important if one wants to prove equality of tropical cycles --- thus providing an important step towards a potential tropicalization statement relating classical and tropical Hurwitz cycles. We will also consider how Hurwitz cycles can be written as divisors of rational functions and how they relate to tropical translations of other representations of algebraic Hurwitz cycles.

Classically, questions about irreducibility and connectedness of Hurwitz spaces have been considered for a long time. Hurwitz \cite{hriemannsche} showed that the space of simple branched covers  of $\mathbb{P}^1$ is connected, using results of Clebsch and Lüroth \cite{ctheorieriemann}. Severi \cite{salgebraische} used this to show that $\curly{M}_g$ is irreducible. These questions become much more difficult however, if one allows target curves of higher genus or more complicated ramification and monodromy - this is a very actively researched topic, see for example \cites{beclassificationbranched, ghsnoteonhurwitz,kirreducibility, virreducibility}.

A very helpful tool in the study of Hurwitz cycles is \atint\footnote{see also \url{https://github.com/simonhampe/atint}}\ \cite{hatint}, an extension for \polymake\footnote{see also \url{www.polymake.org}}\ \cite{gjpolymake}  for tropical intersection theory. With its focus on moduli of curves it provides an easy way to compute examples and a quick method for testing conjectures.

In Section \ref{section_prelim_tropical} we review the basic definitions of tropical geometry. We define tropical varieties and the basic notions of tropical intersection theory. We give a definition of connectedness and irreducibility and discuss their relevance in more detail. We conclude this section with a short introduction to moduli of rational curves and stable maps. In \ref{section_prelim_hurwitz} we define algebraic and tropical Hurwitz cycles. We then look at the latter in more detail, i.e.\ we describe the tropical covers that they parametrize and how a tropical Hurwitz cycle can be computed. In Section \ref{section_connected} we study whether tropical Hurwitz cycles are connected in codimension one. We give a combinatorial proof of the following result:

\begin{theorem*}[Theorem \ref{main_thm_connected}]
 For all $k$, $p$ and $x$, the cycles $\hwt{k}(x,p)$ and $\hw{k}(x,p)$ are connected in codimension one.
\end{theorem*}

In \ref{section_irreducible} we use this to show that all marked Hurwitz cycles are weakly irreducible for a generic choice of simple ramification points:

\begin{corollary*}[Corollary \ref{main_cor_irred}]
 For any $x$ and any pairwise different $p_j$, $\hwt{k}(x,p)$ is weakly irreducible.
\end{corollary*}

We conclude that section with computational examples showing that this is the strongest possible statement.

In \ref{section_hurwitz_cutting} we study how Hurwitz cycles can be cut out by rational functions on $\mk{n}$. We know from \cite{fcocycles} that each subcycle of a matroidal fan (such as $\mk{n}$) can be written as the sum of products of rational functions, but the result is non-constructive. We show that $\hw{n-4}(x,0)$ can be cut out by the rational function that adds up distances of vertex images of covers. To prove this we define the \emph{push-forward} of a rational function under a morphism of equidimensional tropical varieties whose target is smooth.

Finally, in \ref{section_hurwitz_boundary} we consider an alternative representation of the algebraic Hurwitz cycle given in \cite{bcmhurwitz} and its \enquote{tropicalization}. We show that this new tropical cycle is \emph{numerically equivalent} to $\hw{k}(x)$, thus obtaining a strong indicator that our notion of naively tropicalizing is the correct one.

\begin{remark*}
 A paper by Cavalieri, Markwig and Ranganathan \cite{cmrtropicalcompactification}, which appeared shortly after the first submission of this paper, proves that indeed tropical Hurwitz cycles are tropicalizations of algebraic Hurwitz cycles. As a corollary, they obtain the connectedness in codimension one for unmarked Hurwitz cycles (Theorem 3 and Corollary 4). 
\end{remark*}

\begin{acknowledgement}
 I would like to thank Hannah Markwig for many inspiring discussion and the anonymous referees for their helpful suggestions. I was supported by DFG grants MA 4797/3-1 and MA 4797/1-2.
 
 The final publication is available at Springer via http://dx.doi.org/10.1007/s10801-015-0615-0.
\end{acknowledgement}

\section{Preliminaries}\label{section_prelim}

\subsection{Tropical geometry}\label{section_prelim_tropical}

\subsubsection{Weighted polyhedral complexes}

\begin{notation}
 Let $\Lambda$ be a lattice (i.e.\ a finitely generated free abelian group) and $V := \Lambda \tsr_\Z \R$ the associated vector space. We assume all polyhedra in $V$ to be rational, i.e.\ defined by inequalities $g(x) \geq \alpha$ with $g \in \Lambda^\vee$. For a polyhedron $\sigma$ we write $V_\sigma := \gnrt{a-b; a,b \in \sigma}_{\R}$ for the linear part of its affine space and $\Lambda_\sigma := V_\sigma \cap \Lambda$ for its associated lattice.
\end{notation}

\begin{defn}
 A \emph{weighted polyhedral complex} $(\Sigma, \omega)$ is a pure, rational, polyhedral complex $\Sigma$ in $V = \Lambda \tsr_\Z \R$ together with a weight function $\omega$ on its maximal cells, taking values in $\Z$. We write $\abs{\Sigma} := \bigcup_{\sigma \in \Sigma}\sigma$ for the \emph{support} of $\Sigma$.
 
 Let $\sigma$ be a rational $d$-dimensional polyhedron and $\tau$ a face of $\sigma$ of dimension $d-1$ The \emph{lattice normal vector} of $\tau$ with respect to $\sigma$, denoted by $u_{\sigma/\tau}$, is the unique generator of $\Lambda_\sigma / \Lambda_\tau \cong \Z$, such that $g(u_{\sigma/\tau}) > 0$ for all $g \in \Lambda_{\sigma}^\vee$ with $g_{\mid \tau} = 0$ and $g_{\mid \sigma} \geq 0$. By abuse of notation we also write any representative of $u_{\sigma/\tau}$ in $V$ with the same letter.
 
 We call a weighted complex $(\Sigma,\omega)$ \emph{balanced}, if for all codimension one cells $\tau$ the following holds:
 $$\sum_{\sigma > \tau} \omega(\sigma) u_{\sigma/\tau} \in V_\tau.$$
 
 A \emph{tropical cycle} is the equivalence class of a balanced weighted complex modulo refinement, i.e.\ we consider two balanced complexes to be the same, if they have a common refinement respecting the weights. By abuse of notation, we will often use the same letter for a tropical cycle and its polyhedral structure.
 
 A \emph{tropical variety} is a tropical cycle whose weights are greater than zero.
 
 Let $(\Sigma,\omega)$ be a weighted complex and $\tau$ any cell in $\Sigma$. We define the local fan at $\tau$ to be the weighted fan
 $$\Star_\Sigma(\tau) := ( \{ \Pi(\sigma - \tau); \tau \leq \sigma\}, \omega_\Star),$$
 where $\Pi: \R^n \to \R^n / V_\tau$ is the residue map, $\sigma - \tau$ denotes the pointwise difference and the weight function is defined by $\omega_\Star: \Pi(\sigma - \tau) \mapsto \omega(\sigma)$.
 
 The \emph{recession cone} of a polyhedron $\sigma \subseteq V$ is the set
$$\textnormal{rec}(\sigma) := \{v \in V; \exists x \in \sigma \textnormal{ such that } x + \R_{\geq 0} v \subseteq \sigma\}.$$
If $X$ is a tropical cycle, then by \cite{rdiss}*{Lemma 1.4.10} there exists a refinement $\curly{X}$ of its polyhedral structure such that $\delta(X) := \{\textnormal{rec}(\sigma); \sigma \in \curly{X}\}$ is a polyhedral fan (One can use a construction similar to the one used for defining push-forwards). If we define a weight function 
$$\omega_\delta(\textnormal{rec}(\sigma)) := \sum_{\sigma': \textnormal{rec}(\sigma') = \textnormal{rec}(\sigma)} \omega_X(\sigma'),$$
then $(\delta(X),\omega_\delta)$ is a tropical cycle by \cite{rdiss}*{Theorem 1.4.12}.

We call two tropical cycles \emph{rationally equivalent} if $\delta(X) = \delta(Y)$ (up to refinement, of course).
 
 Let $(X,\omega_X)$ be a tropical cycle. A \emph{rational function} on $X$ is a function $\varphi: X \to \R$ that is piecewise affine linear with integer slopes with respect to some polyhedral structure $\curly{X}_\varphi$ of $X$. 
  
  The \emph{divisor} of $\varphi$ is the tropical cycle $\varphi \cdot X := (\curly{Y},\omega_\varphi)$, with $\curly{Y}$ the codimension one skeleton of $\curly{X}_\varphi$ and 
  $$\omega_\varphi(\tau) := \sum_{\sigma > \tau} \omega_X(\sigma) \varphi_\sigma(u_{\sigma/\tau}) - \varphi_\tau \left( \sum_{\sigma > \tau} \omega_X(\sigma) u_{\sigma/\tau} \right),$$
  where $\varphi_\sigma, \varphi_\tau$ denote the linear part of the function restricted to the corresponding cell.
  
    A \emph{morphism} of tropical cycles $f: X \to Y$ is a map from $\abs{X}$ to $\abs{Y}$ which is locally a linear map and respects the underlying lattice, i.e.\ maps $\Lambda_X$ to  $\Lambda_Y$.

The \emph{push-forward} of $X$ is defined as follows: by \cite{gkmmoduli}*{Construction 2.24} there exists a refinement $\curly{X}$ of the polyhedral structure on $X$ such that $\{f(\sigma); \sigma \in \curly{X}\}$ is a polyhedral complex. We then set
$$ f_*(X) = \{f(\sigma); \sigma \in \curly{X}; f \textnormal{ injective on } \sigma\} \;$$
with weights
 $$\omega_{f_*(X)}(f(\sigma)) = \sum_{\sigma': f(\sigma') = f(\sigma)} \abs{\Lambda_{f(\sigma)} / f(\Lambda_{\sigma'})} \omega_X(\sigma').$$
It is shown in \cite{gkmmoduli}*{Proposition 2.25} that this yields a tropical cycle and does not depend on the choice of $\curly{X}$.
 
 If $f: X \to Y$ is a morphism of tropical cycles and $\varphi$ is a rational function on $Y$, then $f^*\varphi = \varphi \circ f$ is the \emph{pull-back} of $\varphi$ via $f$. 
\end{defn}

\subsubsection{Connectedness and irreducibility}

\begin{defn}
 A tropical cycle $X$ is \emph{connected in codimension one}, if for any two maximal cells  $\sigma, \sigma'$ there exists a sequence of maximal cells $\sigma = \sigma_0,\dots,\sigma_r = \sigma'$, such that two subsequent cells $\sigma_i, \sigma_{i+1}$ intersect in codimension one (It is easy to see that this does not depend on the actual choice of polyhedral structure).
 
 We call $X$ \emph{irreducible}, if any $(\dim X)$-dimensional subcycle $Y$ (i.e.\ a tropical cycle with $\abs{Y} \subseteq \abs{X}$) is an integer multiple of $X$.
 
 We call $X$ \emph{weakly irreducible} if $X$ is an integer multiple of an irreducible cycle.
  
\end{defn}

\begin{remark}
 We can measure irreducibility of a tropical cycle $X$ by computing its \emph{weight lattice} $\Omega_X$: this is the lattice of weight functions making it balanced. It has been shown in \cite{hatint} that this does not depend on the choice of polyhedral structure and that $(X,\omega)$ is irreducible if and only if the rank of $\Omega_X$ and the greatest common divisor of all weights $\omega(\sigma)$ are both 1. $\Omega_X$ can be computed as the common solutions of all local balancing equations, which in turn can be interpreted as linear equations in the space of weight functions.

 Somewhat contrary to the terminology, connectedness should probably be considered the \enquote{tropicalization} of irreducibility in the algebraic setting. It was shown in \cite{cpconnectivity} that the tropicalization of any irreducible variety over an algebraically closed field is connected in codimension one. This property is also interesting from a computational point of view: roughly speaking, a connected complex can be computed by starting with a single maximal cell and recursively computing maximal cells that are attached to codimension one faces. This often provides a more efficient approach (see \cite{bjsstcomputingtropical} for an example).
 
 It is not as easy to find an analogue for tropical irreducibility. By \cite{mstropicalbook}*{Theorem 6.7.5}, the weight lattice of a $d$-dimensional complex $\Sigma$ in $\R^n$ is in bijection to $A_{n-d}(X_\Sigma)$. From a purely tropical point of view, irreducibility is a helpful property if you want to show equality of cycles, as one then only needs to prove one inclusion.
 
 Connectedness in codimension one is clearly a necessary condition for irreducibility. Together with local irreducibility we obtain a sufficient criterion:
 
 \begin{prop}[This is an easy generalization of {\cite{rdiss}*{Lemma 1.2.29}}]\label{intro_prop_irred}
  Let $X$ be a tropical cycle. If $X$ is locally (weakly) irreducible (i.e.\ $\Star_X(\tau)$ is (weakly) irreducible for each codimension one face $\tau$) and $X$ is connected in codimension one, then $X$ is (weakly) irreducible.
 \end{prop}

 \end{remark}

 \subsubsection{Tropical rational curves, moduli spaces and Psi classes}
 We only present the basic notations and definitions related to tropical moduli spaces. For more detailed information, see for example \cite{gkmmoduli}.

\begin{defn}
 An $n$-\emph{marked rational tropical curve} is a metric tree with $n$ unbounded edges, labeled with numbers $\{1,\dots,n\}$, such that all vertices of the graph are at least trivalent. We can associate to each such curve $C$ its metric vector $(d(C)_{i,j})_{i < j} \in \R^{\binom{n}{2}}$, where $d(C)_{i,j}$ is the distance between the unbounded edges (called \emph{leaves}) marked $i$ and $j$ determined by the metric on $C$.

Define $\Phi_n: \R^n \to \R^{\binom{n}{2}}, a \mapsto (a_i + a_j)_{i < j}$. Then 
$$\mk{n} := \{d(C); C \textnormal{ $n$-marked curve}\} \subseteq \R^{\binom{n}{2}} / \Phi_n(\R^n)$$
is the \emph{moduli space} of $n$-marked rational tropical curves.
\end{defn}

\begin{remark}\label{remark_moduli_defn}
The space $\mk{n}$ is also known as the \emph{space of phylogenetic trees} \cite{sstropgrass}. It is shown (e.g. in \cite{gkmmoduli}) that $\mk{n}$ is a pure $(n-3)$-dimensional fan and if we assign weight 1 to each maximal cone, it is balanced (though \cite{gkmmoduli} does not use the standard lattice, as we will see below). Points in the interior of the same cone correspond to curves with the same \emph{combinatorial type}: the combinatorial type of a curve is its equivalence class modulo homeomorphisms respecting the labelings of the leaves, i.e.\ morally we forget the metric on each graph. In particular, maximal cones correspond to curves where each vertex is exactly trivalent. We call this particular polyhedral structure on $\mk{n}$ the \emph{combinatorial subdivision}.

The lattice for $\mk{n}$ under the embedding defined above is generated by the rays of the fan. These correspond to curves with exactly one bounded edge. Hence each such curve defines a partition or \emph{split} $I \vert I^c$ on $\{1,\dots,n\}$ by dividing the set of leaves into those lying on the \enquote{same side} of $e$. We denote the resulting ray by $v_I$ (note that $v_I = v_{I^c}$). Similarly, given any rational $n$-marked curve, each bounded edge $E_i$ of length $\alpha_i$ induces some split $I_i \vert I_i^c, i=1,\dots d$ on the leaves. In the moduli space, this curve is then contained in the cone spanned by the $v_{I_i}$ and can be written as $\sum \alpha_i v_{I_i}$. In particular, $\mk{n}$ is a simplicial fan.

There are several reasons why $\mk{n}$ should be considered the tropical analogue of $M_{0,n}$, the algebraic space of rational $n$-marked curves. Perhaps easiest to see is the fact that there is a one-to-one, dimension-reversing relation between combinatorial types of tropical rational curves and boundary strata of $\overline{M}_{0,n}$. Each boundary stratum corresponds to a nodal curve $X$, to which we can assign a \emph{dual graph}. This is a graph which has a vertex for each component of $X$, a bounded edge for each node and an unbounded leaf for each marked point. 

A much stronger relation was proven in \cite{gmequations}, where it is shown that (for the right embedding), the tropicalization of $M_{0,n}$ is $\mk{n}$ and the closure of $M_{0,n}$ in the toric variety $X(\mk{n})$ is $\overline{M}_{0,n}$ (i.e.\ $\overline{M}_{0,n}$ is a \emph{tropical compactification}).

\end{remark}

\begin{defn}\label{intro_def_psiclasses}
 Let $n \geq 3$ and $i \in \{1,\dots,n\}$. The $i$-th \emph{Psi class} is the subset $\psi_i$ of $\mn$, consisting of the locus of all $n$-marked curves such that the $i$-th leaf is attached to a vertex that is at least fourvalent. 
\end{defn}

\begin{remark}
 In the combinatorial subdivision of $\mn$, $\psi_i$ is actually a codimension one sub\emph{fan} and assigning weight 1 to each maximal cone produces a tropical variety. Tropical Psi classes were first defined by Mikhalkin in \cite{mmodulicurves}, as a direct translation of the classical definition. In \cite{kmpsiclasses}, the authors define Psi classes as divisors of rational functions on $\mn$ and give a complete combinatorial description of all products of Psi classes.
\end{remark}

\subsubsection{Tropical stable maps}\label{section_hurwitz_stable}

To study covers of $\R$ in tropical geometry, we will need a tropical space of stable maps. A precise definition can be found in \cite{gkmmoduli}*{Section 4}. For shortness, we will use their result from Proposition 4.7 as definition and explain the geometric interpretation behind it afterwards.

\begin{defn}
 Let $m \geq 4, r \geq 1$. For any $\Delta = (v_1,\dots,v_n), v_i \in \R^r$ with $\sum v_i = 0$ we denote by 
$$\mk{m}(\R^r,\Delta) := \mk{n+m} \times \R^r$$
the \emph{space of stable $m$-pointed maps of degree $\Delta$}.
\end{defn}

\begin{remark}\label{remark_basics_maps}
 An element of $\mk{m}(\R^r,\Delta)$ represents an $(n+m)$-marked abstract curve $C$  together with a continuous, piecewise integer affine linear (with respect to the metric on $C$) map $h: C \to \R^r$. We label the first $n$ leaves by $\{1,\dots,n\}$ and require $h$ to have slope $v_1,\dots,v_n$ on them. The last $m$ leaves we denote by $l_0,\dots,l_{m-1}$. These are contracted to a point under $h$. Since we want the image curve to be a tropical curve in $\R^r$, the slope on the bounded edges is already uniquely defined by the condition that the outgoing slopes of $h$ at each vertex have to add up to 0. This defines the map $h$ up to a translation in $\R^r$. The translation is fixed by the $\R^r$-coordinate, which can for example be interpreted as the image of the first contracted end $l_0$ under $h$ (see figure  \ref{prelim_fig_stable} for an example). There are obvious evaluation maps $\ev_i: \mk{m}(\R^r,\Delta) \to \R^r, i = 0,\dots,m-1$, mapping a stable map to $h(l_i)$. \cite{gkmmoduli}*{Proposition 4.8} shows that these are morphisms. Similarly, there is a \emph{forgetful morphism} $\ft: \mk{m}(\R^r,\Delta) \to \mk{n}$, forgetting the contracted ends and the map $h$. 

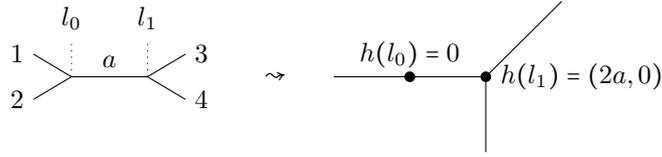
\begin{figure}[ht]
 \centering
 \begin{tikzpicture}
  \matrix[column sep = 5mm]{
	\draw (0,0) -- (1,0);
	\draw (0.5,0) node[above]{$a$};
	\draw[dotted] (0,0) -- (0,0.5) node[above]{$l_0$};
	\draw[dotted] (1,0) -- (1,0.5) node[above]{$l_1$};
	\draw (0,0) -- (-0.5,0.3) node[left]{1};
	\draw (0,0) -- (-0.5,-0.3) node[left]{2};
	\draw (1,0) -- (1.5,0.3) node[right]{3};
	\draw (1,0) -- (1.5,-0.3) node[right]{4}; &
	\draw (0,0) node{$\leadsto$}; &
	\draw (0,0) -- (-2,0);
	\draw (0,0) -- (1,1);
	\draw (0,0) -- (0,-1);
	\fill[black] (-1,0) circle (2pt) node[above]{$h(l_0) = 0$};
	\fill[black] (0,0) circle (2pt) node[right = 2pt]{$h(l_1) = (2a,0)$};
	\\
      };
 \end{tikzpicture}
\caption{On the left the abstract 6-marked curve $\Gamma = a\cdot v_{\{1,2,l_0\}}$. If we pick $\Delta = ( (-1,0),(-1,0),(2,2), (0,-2))$ and fix $h(l_0) = 0$ in $\R^2$, we obtain the curve on the right hand side as $h(\Gamma)$.}\label{prelim_fig_stable}
\end{figure}
\end{remark}

% To mimic the classical situation for Hurwitz cycles, we will want to consider stable maps to $\R$. Since the image of each stable map is a tropical variety, stable maps of $\R$ are automatically covers. As a technical simplification, we also want to identify tropical covers that only differ by a translation of the target $\R$. To that end, we make the following definition:
% 
% \begin{defn}
%  Let $m \geq 4$. We define the \emph{space of rational $m$-marked covers of degree $\Delta \in \R^n$ of a parametrized $\R$} to be the tropical variety 
%  $$\mk{m}(\R^{\textnormal{par}},\Delta) := \mk{n+m} \times \{0\} \subseteq \mk{m}(\R,\Delta).$$
% \end{defn}
% 
% \begin{remark}
%  Geometrically speaking, we can interpret this definition in the following manner: We fix the translation of the target $\R$ by requiring the first contracted end to be mapped to 0. As a subspace of $\mk{m}(\R,\Delta)$, we still have well-defined evaluation morphisms $\ev_0,\dots,\ev_{m-1}$ to $\R$ (note that $\ev_0 \equiv 0$ by our convention) and a forgetful map to $\mk{n}$. Also, since our space of covers is obviously isomorphic to $\mk{n+m}$, we have tropical Psi-classes (see Section \ref{section_moduli_psi}).
% \end{remark}
\newpage
\subsection{Hurwitz cycles}\label{section_prelim_hurwitz}

\subsubsection{Algebraic Hurwitz cycles}\label{section_hurwitz_algebraic}

We will only briefly cover algebraic Hurwitz cycles, as we will be working exclusively on the tropical side. For a more in-depth discussion of its definition and properties, see for example \cites{bcmhurwitz, gvrelativevirtual}.

Let $n \geq 4$. We define 

$$\curly{H}_n := \left\{x \in \Z^n: \sum_{i=1}^n x_i = 0 \right\} \wo \{0\}.$$

Let $x \in \curly{H}_n$ and choose distinct points $p_0,\dots,p_{n-3-k} \in \P^1 \wo \{0,\infty\}$. The \emph{double Hurwitz cycle} $\mathbb{H}_k(x)$ is a $k$-dimensional cycle in the moduli space of rational $n$-marked curves $\overline{M}_{0,n}$. It parametrizes curves $C$ that \emph{allow} covers $C \stackrel{\pi}{\to} \P^1$ with the following properties:
\begin{itemize}
 \item $C$ is a smooth connected rational curve.
 \item $\pi$ has ramification profile $x^+ := (x_i; x_i > 0)$ over 0 and ramification profile $x^- := (x_i; x_i < 0)$ over $\infty$. The corresponding ramification points are the marked points of $C$.
 \item $\pi$ has simple ramification over the $p_i$ and at most simple ramification elsewhere.
\end{itemize}

The precise definition \cite{bcmhurwitz}*{Section 3} actually involves some moduli spaces. For the sake of simplicity, we will just cite the following result, that can be taken as a definition throughout this paper.

\begin{lemma}[{\cite{bcmhurwitz}*{Lemma 3.2}}]\label{intro_hurwitz_def}
 $$\mathbb{H}_k(x) = st_*\left( \prod_{i=1}^{n-2-k} \psi_i \ev_i^*([pt])\right),$$
where 
\begin{itemize}
 \item the intersection product is taken in $\overline{M}_{0,n-2-k}(x)$, the space of relative stable maps to $\P^1$ with ramification profile $x^+,x^-$ over $0$ and $\infty$ (see also \cite{gvrelativevirtual} for a definition. In their language, this is the space of maps to a rigid target).
 \item $st: \overline{M}_{0,n-2-k}(x) \to \overline{M}_{0,n}$ is the morphism forgetting the map and all marked points but the ramification points over $0$ and $\infty$ (and stabilizing the result by contracting components that become unstable, i.e.\ contain less than three special points).
\end{itemize}

\end{lemma}

\subsubsection{Tropical Hurwitz cycles}\label{section_hurwitz_tropical}

We already have all ingredients at hand to \enquote{tropicalize} Lemma \ref{intro_hurwitz_def}. Note that a point $q \in \R$ can be considered as the divisor of the tropical polynomial $\max\{x,q\}$, so it can be pulled back along a morphism to $\R$. Also, as $\mk{n+m}$ is a subcycle of $\mk{m}(\R^r,\Delta) = \mk{n+m} \times \R^r$, we can define Psi classes on the latter: for $i = 0,\dots m-1$, we define 
$$\Psi_i := \psi(l_i) \times \R^r,$$
where $\psi(l_i)$ is the Psi class of $\mk{n+m}$ associated to the leaf $l_i$ we defined in \ref{intro_def_psiclasses}. 

\begin{defn}[{\cite{bcmhurwitz}*{Definition 6}}]
 Let $x \in \Z^n \wo \{0\}$ with $\sum x_i = 0, k \geq 0$ and $N := n-2-k$. Choose $p := (p_0,\dots,p_{N-1}), p_i \in \R$. We define the \emph{tropical marked Hurwitz cycle}
$$\hwt{k}(x,p) := \left(\prod_{i=0}^{N-1} (\Psi_i \ev_i^*(p_i))\right)\cdot \mk{N}(\R,x) $$
We then define the \emph{tropical Hurwitz cycle} 
$$\hw{k}(x,p) := \ft_*(\hwt{k}(x,p)) \subseteq \mk{n}.$$
\end{defn}

\begin{remark}
In \cite{cjmhurwitz} the authors show that Hurwitz numbers can be considered as a weighted count of tropical covers of $\R$, which are monodromy graphs of algebraic covers. In particular, the ramification profile over 0 and $\infty$ appears on the tropical side as the slopes of the ends going to $\pm \infty$. Thus a tropical analogue of a cover with prescribed ramification profile $x$ is an element of $\mk{N}(\R,x)$. Hence the above definition becomes the exact analogue of Lemma \ref{intro_hurwitz_def} and gives us $k$-dimensional tropical cycles $\hwt{k}(x,p), \hw{k}(x,p)$. While it formally depends on the choice of the $p_j$, two different choices $p,p'$ lead to rationally equivalent cycles $\hw{k}(x,p) \sim \hw{k}(x,p')$. The reason for this is that any two points in $\R$ are rationally equivalent and this is compatible with pullbacks and taking intersection products. In particular, if we choose all $p_i$ to be equal (e.g.\ equal to $0$), we obtain fans, which we denote by $\hwt{k}(x)$ and $\hw{k}(x)$. They are obviously the recession fans of $\hwt{k}(x,p),\hw{k}(x,p)$ for any $p$.
\end{remark}

\begin{ex}\label{hurwitz_ex_cover}
 Let us now see what kind of object these Hurwitz cycles represent. As discussed in Remark \ref{remark_basics_maps}, for any fixed $x$ and any $n$-marked curve $C$ we obtain a map $h: C \to \R$ up to translation. To determine such a map, we have to fix an orientation of each edge and leaf of $C$ and an integer slope along this orientation. In informal terms, the orientation determines how we position an edge or leaf on $\R$ (the \enquote{tip} of the arrow points towards $+ \infty$). The slope can then be seen as a stretching factor.

The orientation of each leaf $i$ is chosen so that it \enquote{points away} from its vertex if and only if $x_i > 0$. We define its slope to be $\abs{x_i}$. Any bounded edge $e$ induces a split $I_e$. Its slope is $\abs{x_{I_e}}$, where $x_{I_e} = \sum_{i \in I_e} x_i$. We pick the orientation such that at each vertex the sum of slopes of incoming edges is the sum of slopes of outgoing edges (it is not hard to see that such an orientation exists and must be unique).

As we discussed before, we can fix the translation of $h$ by requiring the image of any of its vertices $q$ to be some $\alpha \in \R$. Denote by $h(C,q,\alpha): C \to \R$ the corresponding map. Figure \ref{figure_intro_cover} gives two examples of this construction.
% 
%  For example, let $x = (1,1,1,1,-4), C = \frac{1}{2}v_{\{1,2\}} + \frac{1}{3}v_{\{4,5\}}$ and $C' = v_{\{1,2\}} + v_{\{3,4\}}$. Let $q$ be the vertex adjacent to the third leaf in $C$ and $q'$ the vertex at the fifth leaf in $C'$. Set $\alpha = \alpha' = 0$. Then $c(C,q,\alpha)$ and $c(C',q',\alpha')$ are depicted in Figure \ref{figure_intro_cover}.
% 
 \begin{figure}[ht]
  \centering
 \begin{tikzpicture}
   \matrix[column sep = 2mm,row sep = 5mm, ampersand replacement=\&]{
    \draw (-1,0) node[left = 30pt]{$C := $};
    \draw[font=\itshape,<-, shorten <= 2pt] (-1,0) -- (0,0);
    \draw[font=\itshape,<-, shorten <= 2pt] (0,0) -- (1,0) ;
    \draw[->] (-1,0)  -- (-1.5,0.3) node[left]{1};
    \draw (-0.5,0) node[above, font=\itshape] {\tiny $\substack{l = 1/2 \\ \omega = 2}$};
    \draw (0.5,0) node[above,font=\itshape] {\tiny $\substack{l = 1/3 \\ \omega = 3}$};
    \draw[->] (-1,0) -- (-1.5,-0.3) node[left]{2};
    \draw[->] (0,0) -- (0,0.5) node[above]{3};
    \draw[->] (1,0) -- (1.5,0.3) node[right]{4};
    \draw[<-, shorten <= 2pt] (1,0) -- (1.5,-0.3) node[right]{5}; 
    \fill[red] (0,0) circle (2pt) node[below = 5pt] {$q$}; 
    \fill[darkgreen] (-1,0) circle (2pt) node[below = 5pt] {$v_1$}; 
    \fill[blue] (1,0) circle (2pt) node[below = 5pt] {$v_2$}; \&
    \draw[font=\itshape,<-, shorten <= 2pt] (-1,0)  -- (0,0);
    \draw[font=\itshape,->, shorten >= 2pt] (0,0) -- (1,0) ;
    \draw (-0.5,0) node[above,font=\itshape] {\tiny $\substack{l = 1 \\ \omega = 2}$};
    \draw (0.5,0) node[above,font=\itshape] {\tiny $\substack{l = 1 \\ \omega = 2}$};
    \draw[->] (-1,0) -- (-1.5,0.3) node[left]{1};
    \draw[->] (-1,0) -- (-1.5,-0.3) node[left]{2};
    \draw[<-,shorten <= 2pt]  (0,0) -- (0,0.5) node[above]{5};
    \draw[->]  (1,0) -- (1.5,0.3) node[right]{3};
    \draw[->]  (1,0) -- (1.5,-0.3) node[right]{4}; 
    \fill[darkgreen] (0,0) circle (2pt) node[below = 5pt] {$v_1$};
    \fill[red] (-1,0) circle (2pt) node[below = 5pt] {$q$}; 
    \fill[blue] (1,0) circle (2pt) node[below = 5pt] {$v_2$};
    \draw (1,0) node[right = 30pt]{$ =: C'$};
    \\

    \draw (-2,0) -- (2,0);
    \fill[black] (0,0) circle (2pt) node[below]{\tiny $0$};
%     \draw[text = green] (0,0) node[below = 10]{Right};
    \fill[black] (1,0) circle (2pt) node[below]{\tiny $1$};
    \fill[black] (-1,0) circle (2pt) node[below]{\tiny -1};
    \draw (-1,1) -- (1,1); 
%       \fill[black] (1,1) circle (2pt) node[above]{\tiny $l_7$};
    \draw (1,1) -- (2,1.1) node[right]{\tiny 1};
    \draw (1,1) -- (2,0.9) node[right]{\tiny 2};
    \draw (0,1) -- (2,0.7) node[right]{\tiny 3};
    \draw (-1,1) -- (2,0.5) node[right]{\tiny 4};
    \draw (-1,1) -- (-2,1.1) node[left]{\tiny 5}; 
    \fill[red] (0,1) circle (2pt) node[above]{\tiny $q$};
    \fill[blue] (-1,1) circle (2pt) node[above]{\tiny $v_2$};
    \fill[darkgreen] (1,1) circle (2pt) node[above]{\tiny $v_1$};
    \&
    \draw (-2,0) -- (2,0);
%     \draw[text = red] node[below = 10] {Wrong};
    \fill[black] (0,0) circle (2pt) node[below]{\tiny -1};
    \fill[black] (1,0) circle (2pt) node[below]{\tiny $0$};
    \fill[black] (-1,0) circle (2pt) node[below]{\tiny -2};
%     \fill[black, text = red] (1,0.5) circle (2pt) node[above]{\tiny $l_7$};
    \draw (1,1) -- (2,1.1) node[right]{\tiny 1};
    \draw (1,1) -- (2,0.9) node[right]{\tiny 2};
    \draw (1,0.5) -- (2,0.6) node[right]{\tiny 3};
    \draw (1,0.5) -- (2,0.4) node[right]{\tiny 4};
    \draw (1,1) -- (-1,0.75) -- (1,0.5);
    \draw (-1,0.75) -- (-2,0.75) node[left]{\tiny 5};
    \fill[red] (1,1) circle (2pt) node[above]{\tiny $q$};
    \fill[blue] (1,0.5) circle (2pt) node[above]{\tiny $v_2$};
    \fill[darkgreen] (-1,0.75) circle (2pt) node[above]{\tiny $v_1$};
\\
   };
 \end{tikzpicture}
 \caption{The covers defined by two 5-marked rational curves after fixing the image of a vertex $q$ to be $\alpha = 0$. We chose $x = (1,1,1,1,-4)$ and denoted edge lengths by $l$, edge slopes by $\omega$. }\label{figure_intro_cover}
\end{figure}
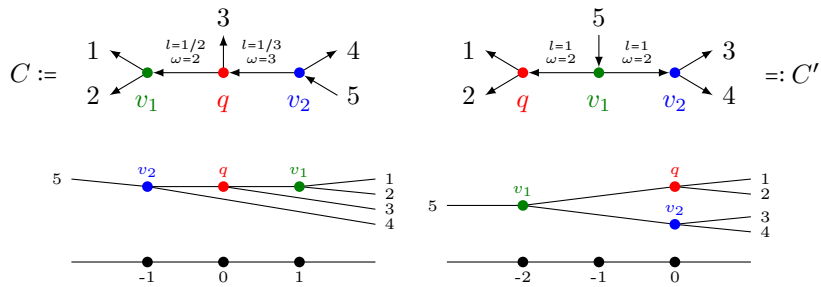

Now choose $p_0,\dots,p_{N-1} \in \R$. Then $\hw{k}(x,p)$ is (set-theoretically) the set of all curves $C$, where we can find vertices $q_0, \dots,q_{N-1}$ (each vertex $q$ can be picked a number of times equal to $\val(q) -2$), such that $h(C,q_0,p_0)(q_l) = p_l$ for all $l$, i.e.\ all curves that \emph{allow} a cover with fixed images for some of its vertices. E.g.\ in Figure \ref{figure_intro_cover}, we have
\begin{itemize}
 \item $C \in \hw{1}(x, p = (0,1))$, but $C \notin \hw{1}(x, p = (0,0))$.
 \item $C' \in \hw{1}(x, p = (0,0))$, but $C' \notin \hw{1}(x, p = (0,1))$.
\end{itemize}
In particular, if we choose $p_i = 0$ for all $i$, $\hw{k}(x,p)$ is the set of all curves, such that $n-2-k$ of its vertices have the same image (again, counting higher-valent vertices $v$ with multiplicity $\val(v)-2$).

Of course there may be several possible choices of vertices that are compatible with $p$. In $\hwt{k}(x,p)$, we fix a choice by attaching the contracted end $l_i$ to the vertex we wish to be mapped to $p_i$. I.e.\ $\hwt{k}(x,p)$ is the set of all curves $C$, such that $l_0,\dots,l_{N-1}$ are attached to vertices and such that in the corresponding cover the vertex with leaf $l_i$ is mapped to $p_i$. For example, in Figure \ref{figure_intro_cover} on the left hand side there are two possible choices of vertices that are compatible with $p = (0,1)$. Hence there are two preimages in $\hwt{k}(x,p)$ corresponding to attaching the contracted leaves $l_0,l_1$ either to $q$ and $v_1$ or to $v_2$ and $q$.
\end{ex}

\begin{remark}\label{hurwitz_remark_weights}
 Let us see how the weight of a cell of $\hw{k}(x)$ is computed if we choose the $p_i$ to be generic, i.e.\ pairwise different. Let $\tau$ be a maximal cell of $\hw{k}(x)$ and $C$ the curve corresponding to an interior point of $\tau$. Then $\tau$ must lie in the interior of a maximal cell $\sigma$ of $\mk{n}$ and for a generic choice of $C$ there is a unique choice of vertices $q_0,\dots,q_{N-1}$ compatible with the $p_i$ (which fixes a cover). Marking these vertices accordingly, we can consider $\sigma$ as a cone in $\mk{N}(\R,x)$. We thus obtain well-defined and linear evaluation maps $\ev_i: \sigma \to \R$, mapping each curve in $\sigma$ to the image of the vertex $q_i$. Assume $\sigma$ is spanned by the rays $v_{I_1},\dots,v_{I_{n-3}}$, then we can write $\ev_i$ in the coordinates of these rays as $(a_1^i,\dots,a_{n-3}^i)$, where $a_k^i = \ev_i(v_{I_k})$. It is shown in \cite{bcmhurwitz}*{Lemma 4.4} that the weight of $\tau$ is then the greatest common divisor of the maximal minors of the matrix $(a_k^i)_{k,i}$.

 In the case that all $p_i$ are 0, we use the fact that $\hw{k}(x,p)$ is the recession fan of the Hurwitz cycle obtained for a generic choice of $p_i$. By its definition this means that the total weight of a cell $\tau$ is obtained as $$\omega(\tau) = \sum_{\tau \subseteq \sigma} \sum_{q_i} g_{\sigma,q_i},$$
where the first sum runs over all maximal cones $\sigma$ of $\mk{n}$ containing $\tau$, the second sum runs over all vertex choices $q_0,\dots q_{N-1}$ that are compatible in $\sigma$ with a \emph{generic} choice of $p_i$ and $g_{\sigma,q_i}$ is the gcd we obtained in the previous construction. In fact, one can easily see that the same method can be used for computing weights if only some of the $p_i$ are equal. 
\end{remark}

\subsection{Computation}\label{section_hurwitz_comp}

If we approach this naively, we already have everything at hand to compute at least marked Hurwitz cycles with \atint: \cite{hatint} tells us how to compute a product of Psi classes (without having to compute the ambient moduli space, which will be huge!) and then we only have to compute divisors of tropical polynomials on this product. However, this only works for small $k$, i.e.\ large codimension. Otherwise, the Psi class product will already be too large to make this computation feasible.

Also, we will mostly be interested in unmarked Hurwitz cycles and computing push-forwards is, computationally speaking, not desirable. One has to produce a very fine polyhedral structure to make sure that the images of the cones form a fan. The following approach to compute unmarked cycles directly proves to be more suitable:

Assume we want to compute $\hw{k}(x,p = (p_0,\dots,p_{N-1}))$ for $x \in \Z^n$. Fix a combinatorial type $C$ of a threevalent rational $n$-marked curve, i.e.\ a maximal cone $\sigma$ of $\mk{n}$. For each choice of distinct vertices $q_0,\dots,q_{N-1}$ of $C$, we obtain linear evaluation maps on $\sigma$, by considering it as a cone of stable maps, where the additional marked ends are attached to the $q_i$. We can now refine $\sigma$ by intersecting it with the fan $F_i$, whose maximal cones are $$F_i^+ := \{x \in \sigma: \ev_i(x) \geq p_i\}, F_i^- := \{x \in \sigma: \ev_i(x) \leq p_i\}.$$
Iterating over all possible choices of $q_i$, this will finally give us a subdivision $\sigma'$ of $\sigma$. The part of $\hw{k}(x,p)$ that lives in $\sigma$ is now a subcomplex of the $k$-skeleton of $\sigma'$: it consists of all $k$-dimensional cells $\tau$ of $\sigma'$ such that there exists a choice of vertices $q_i$ with the property that the corresponding evaluation maps fulfill $\ev_i(x) = p_i$ for all $x \in \tau$. The weight of such a $\tau$ can then be computed using the method described in Remark \ref{hurwitz_remark_weights}.

The full Hurwitz cycle can now be computed by iterating over all maximal cones of $\mk{n}$. This gives a feasible algorithm at least for $n \leq 8$ - after that, the moduli space itself becomes too large.

\begin{ex}
 We want to compute (part of) a Hurwitz cycle: we choose $k = 2,x = (2,2,6,-5,-4,-1)$ and $(p_0,p_1) = (0,1)$. Since the complete cycle would be rather large and difficult to visualize (3755 maximal cells living in $\R^9$), we only consider the part of $\hw{2}(x,p)$ lying in the three-dimensional cone of $\mk{6}$ corresponding to the combinatorial type
 $$C = v_{\{1,2\}} + v_{\{4,5,6\}} + v_{\{5,6\}}.$$
 Figure \ref{hurwitz_fig_subdiv} shows the corresponding cover, together with the part of the Hurwitz cycle we computed using the method described above. Each cell of the cycle is obtained by choosing specific vertices of $C$ for the additional marked points $p_0$ and $p_1$. The correspondence between these choices and the actual cells, together with the corresponding equation, is laid out in Figure \ref{hurwitz_fig_choices}. While there are of course in theory $4 \cdot 4 = 16$ possible choices, not all of them produce a cell: we only display choices of distinct vertices, such that the image of the vertex for $p_1 = 1$ is larger than the image of the vertex for $p_0 = 0$. This gives $\binom{4}{2} = 6$ valid choices.
 
 \begin{figure}[ht]
  \centering
  \begin{tabular}{c c}
   \includegraphics[scale = 0.25]{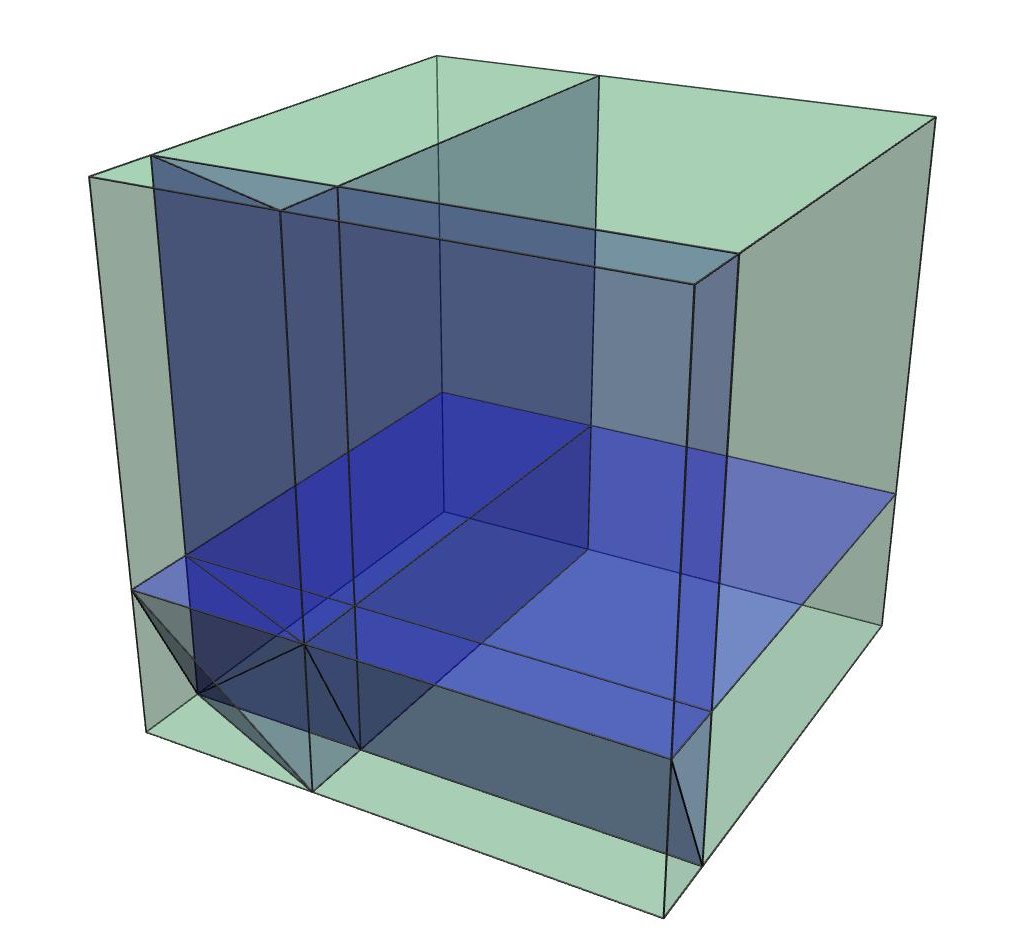} &
       \begin{tikzpicture}
   \draw[->] (0,0,0) -- (1,-0.3,0) node[right]{$\alpha$};
   \draw[->] (0,0,0) -- (0.5,0.5,0) node[right]{$\beta$};
   \draw[->] (0,0,0) -- (0,1,0) node[above] {$\gamma$};
   \draw[color = white] (0,0,0) -- (0,-3,0);
  \end{tikzpicture} 
   \\ 
   \noalign{\smallskip}
   \multicolumn{2}{c}{
   \begin{tikzpicture}
   \matrix[row sep=2mm,column sep=1.5mm,ampersand replacement=\&] {
      \draw (0,0) -- (1,0) -- (2,0) -- (3,0);
      \draw (0,0) -- (-0.5,-0.3) node[left]{5};
      \draw (0,0) -- (-0.5,0.3) node[left]{6};
      \draw (1,0) -- (1,0.5) node[above]{4};
      \draw (2,0) -- (2,0.5) node[above]{3};
      \draw (3,0) -- (3.5,0.3) node[right]{1};
      \draw (3,0) -- (3.5,-0.3) node[right]{2}; 
      \draw (0.5,0) node[below]{$\gamma$};
      \draw (1.5,0) node[below]{$\beta$};
      \draw (2.5,0) node[below]{$\alpha$};
      \&
      \draw[->] (0,0) -- (0.5,0); \&
      \draw (0,0) -- (1,0) -- (3,0) -- (3.8,0);
      \draw (0,0) -- (-1,0) node[left]{5};
      \draw (0,0) -- (-1,-0.3) node[left]{6};
      \draw (1,0) -- (-1,0.3) node[left]{4};
      \draw (3,0) -- (4.8,-0.3) node[right]{3};
      \draw (3.8,0) -- (4.8,0) node[right]{2};
      \draw (3.8,0) -- (4.8,0.3) node[right]{1};  
      \draw (0.5,0) node[below]{$5\gamma$};
      \draw (2,0) node[below]{$10\beta$};
      \draw (3.4,0) node[above]{$4\alpha$};
      \\
%    \fill[black] (0,0) circle (2pt);
%    \fill[black] (1,0) circle (2pt);
    };
  \end{tikzpicture}}
  \end{tabular}
  \caption{The cube represents the three-dimensional cone in $\mk{6}$ that corresponds to the combinatorial type $v_{\{1,2\}} + v_{\{4,5,6\}} + v_{\{5,6\}}$ drawn on the bottom left part of the picture. We denote the length of the interior edges by $\alpha, \beta, \gamma$ as indicated. The blue cells represent the Hurwitz cycle living in this cone. The bottom right figure indicates the corresponding cover. The parameters we chose here are $k = 2, x = (2,2,6,-5,-4,-1)$ and $(p_0,p_1) = (0,1)$.}\label{hurwitz_fig_subdiv}
 \end{figure}
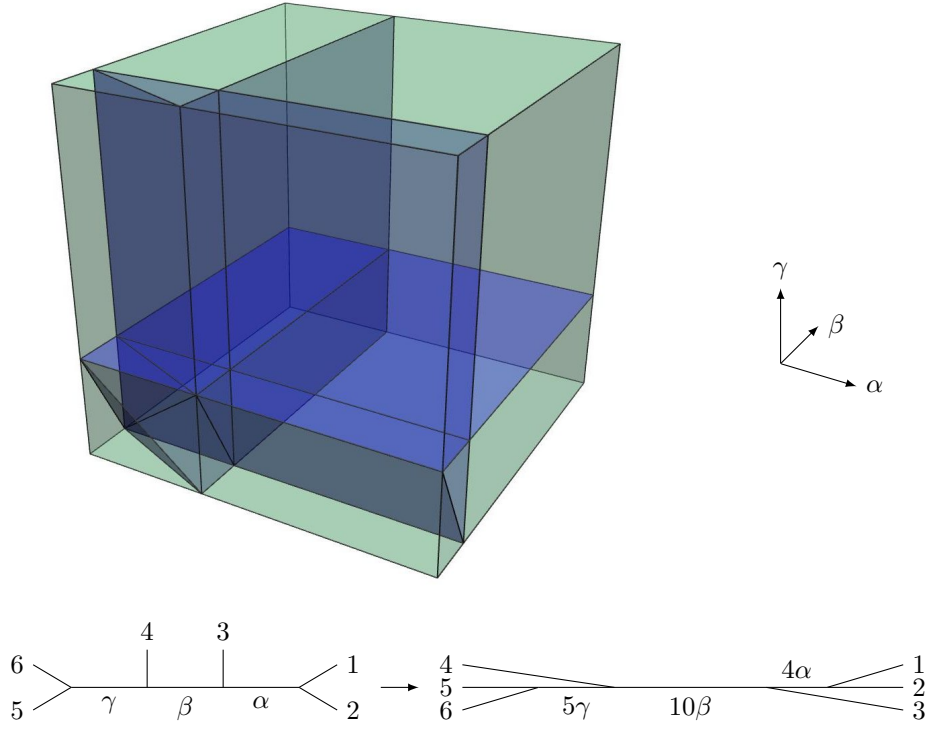

\end{ex}

\begin{figure}[p]
\centering

 \begin{tabular}{|c | c |}
%   \hline
%   \multicolumn{3}{|r|}{\includegraphics[scale=0.2]{hurwitz_small_total}
%     \begin{tikzpicture}
%    \draw[->] (0,0,0) -- (1,-0.3,0) node[right]{$v_{\{1,2\}}$};
%    \draw[->] (0,0,0) -- (0.5,0.5,0) node[right]{$v_{\{4,5,6\}}$};
%    \draw[->] (0,0,0) -- (0,1,0) node[above] {$v_{\{5,6\}}$};
%   \end{tikzpicture}
%   } \\
  \hline
  
  \includegraphics[scale=0.14]{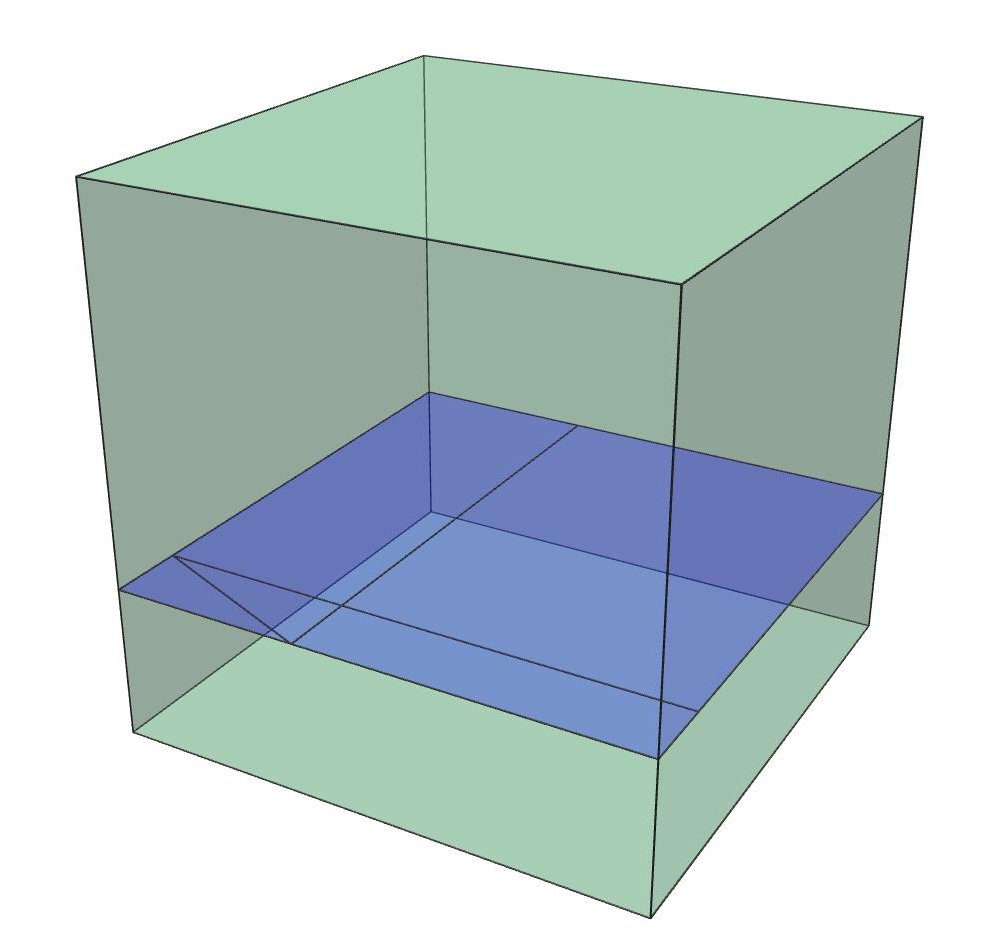} & \includegraphics[scale=0.14]{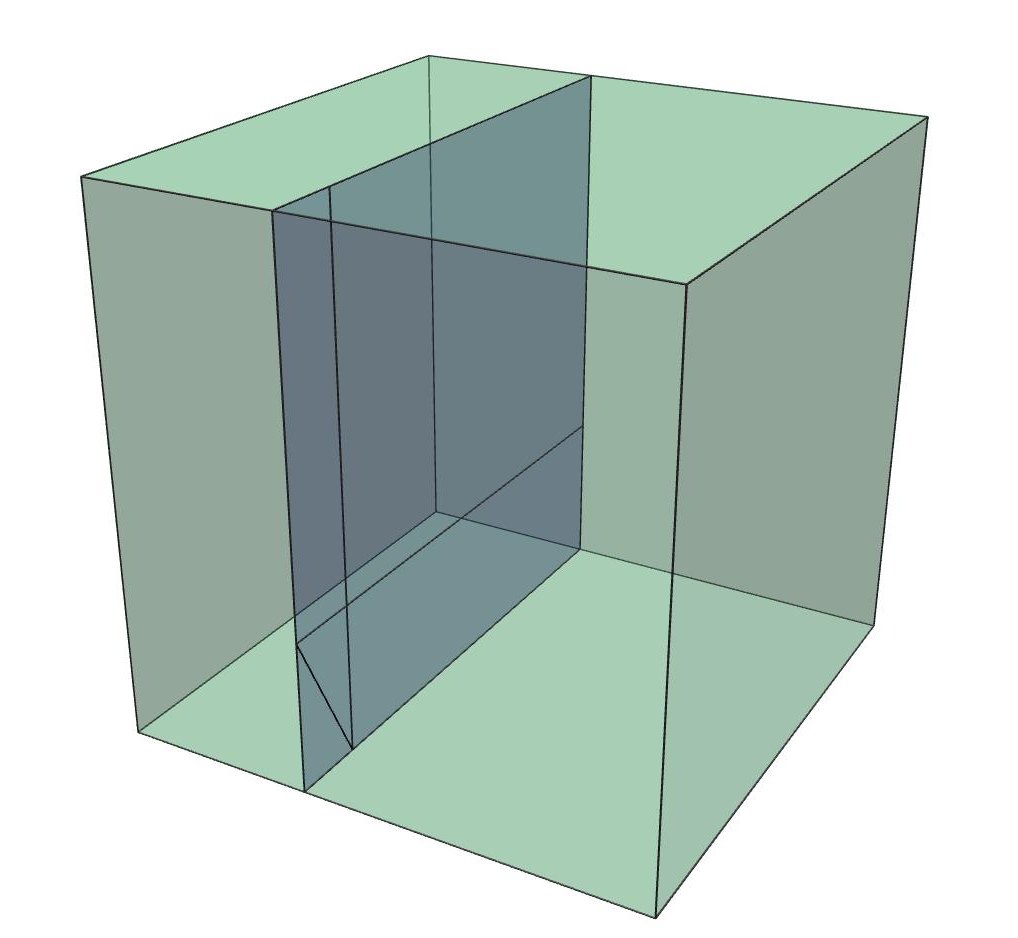}  \\
  \begin{tikzpicture}
   \draw (0,0) -- (1,0) -- (2,0) -- (3,0);
   \draw (0,0) -- (-1,0) node[left]{5};
   \draw (0,0) -- (-1,-0.3) node[left]{6};
   \draw (1,0) -- (-1,0.3) node[left]{4};
   \draw (2,0) -- (4,-0.3) node[right]{3};
   \draw (3,0) -- (4,0) node[right]{2};
   \draw (3,0) -- (4,0.3) node[right]{1};
   \fill[black] (0,0) circle (2pt) node[below]{$p_0$};
   \fill[black] (1,0) circle (2pt) node[below]{$p_1$};
  \end{tikzpicture} &
  \begin{tikzpicture}
   \draw (0,0) -- (1,0) -- (2,0) -- (3,0);
   \draw (0,0) -- (-1,0) node[left]{5};
   \draw (0,0) -- (-1,-0.3) node[left]{6};
   \draw (1,0) -- (-1,0.3) node[left]{4};
   \draw (2,0) -- (4,-0.3) node[right]{3};
   \draw (3,0) -- (4,0) node[right]{2};
   \draw (3,0) -- (4,0.3) node[right]{1};
   \fill[black] (2,0) circle (2pt) node[below]{$p_0$};
   \fill[black] (3,0) circle (2pt) node[above]{$p_1$};
  \end{tikzpicture} \\
  $5\gamma = 1$ & $4\alpha = 1$ \\
  & \\
  \hline
  \includegraphics[scale=0.14]{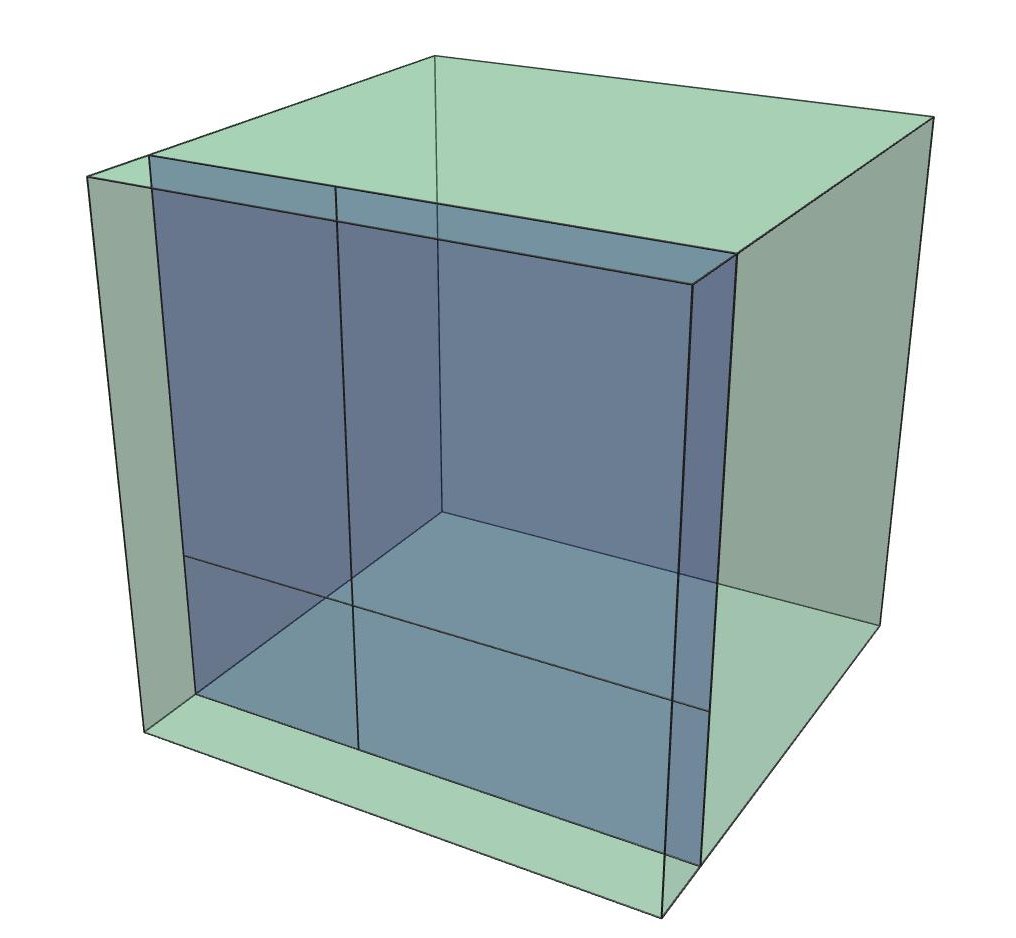}  &
  \includegraphics[scale=0.14]{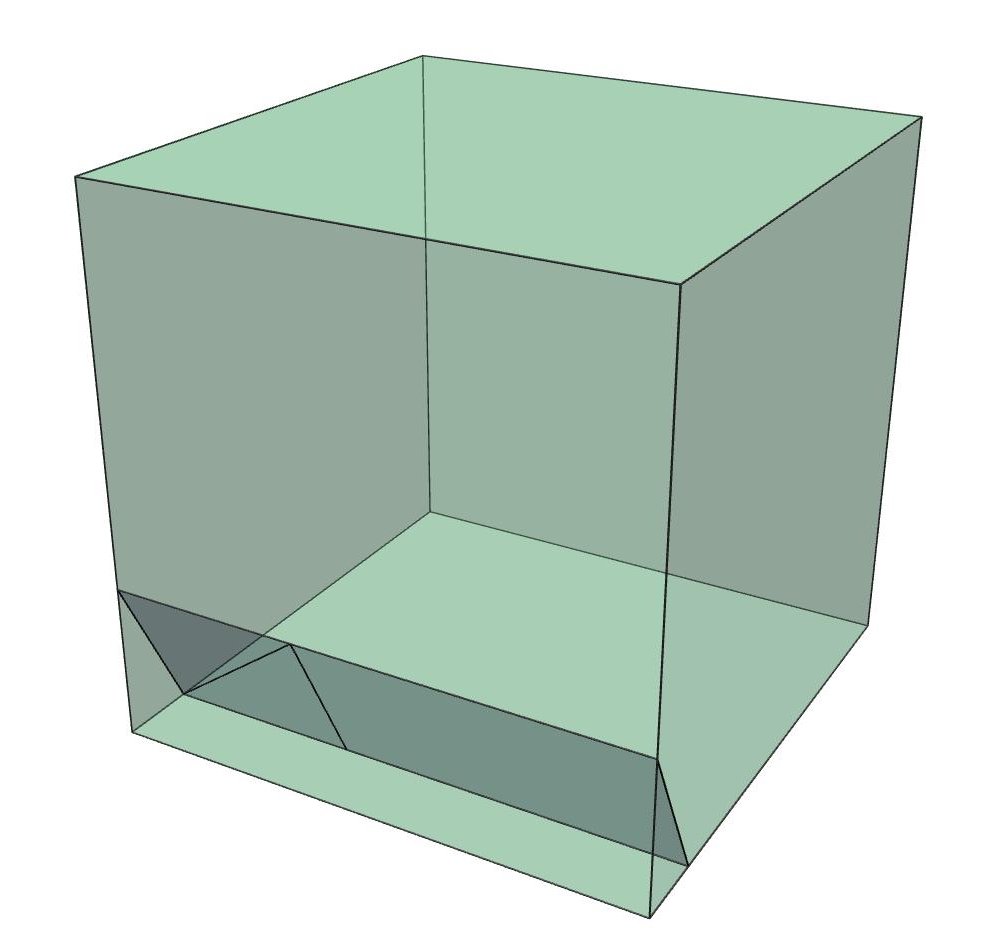}  \\
  
 \begin{tikzpicture}
   \draw (0,0) -- (1,0) -- (2,0) -- (3,0);
   \draw (0,0) -- (-1,0) node[left]{5};
   \draw (0,0) -- (-1,-0.3) node[left]{6};
   \draw (1,0) -- (-1,0.3) node[left]{4};
   \draw (2,0) -- (4,-0.3) node[right]{3};
   \draw (3,0) -- (4,0) node[right]{2};
   \draw (3,0) -- (4,0.3) node[right]{1};
   \fill[black] (2,0) circle (2pt) node[above]{$p_1$};
   \fill[black] (1,0) circle (2pt) node[below]{$p_0$};
  \end{tikzpicture} 2 & 
  \begin{tikzpicture}
   \draw (0,0) -- (1,0) -- (2,0) -- (3,0);
   \draw (0,0) -- (-1,0) node[left]{5};
   \draw (0,0) -- (-1,-0.3) node[left]{6};
   \draw (1,0) -- (-1,0.3) node[left]{4};
   \draw (2,0) -- (4,-0.3) node[right]{3};
   \draw (3,0) -- (4,0) node[right]{2};
   \draw (3,0) -- (4,0.3) node[right]{1};
   \fill[black] (0,0) circle (2pt) node[below]{$p_0$};
   \fill[black] (2,0) circle (2pt) node[above]{$p_1$};
  \end{tikzpicture} \\
  $10\beta = 1$ & $5\gamma + 10\beta = 1$\\
  & \\
  \hline
  \includegraphics[scale=0.14]{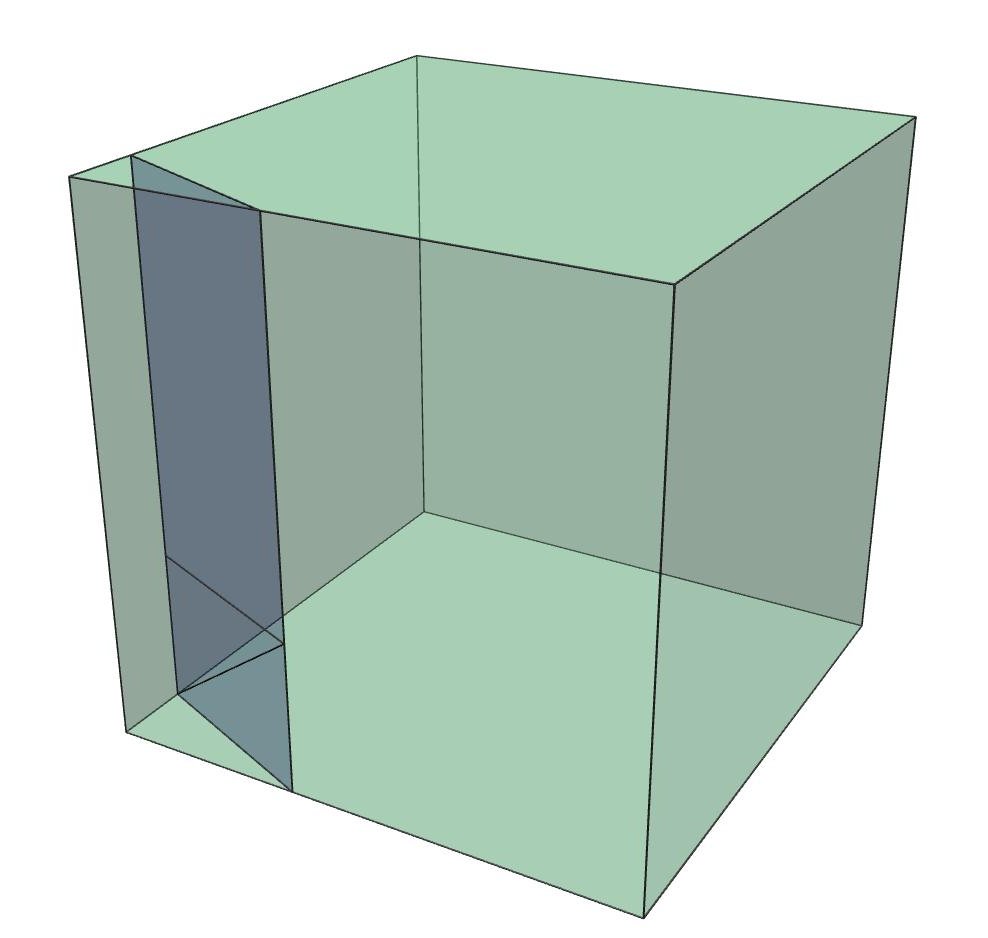}  & \includegraphics[scale=0.14]{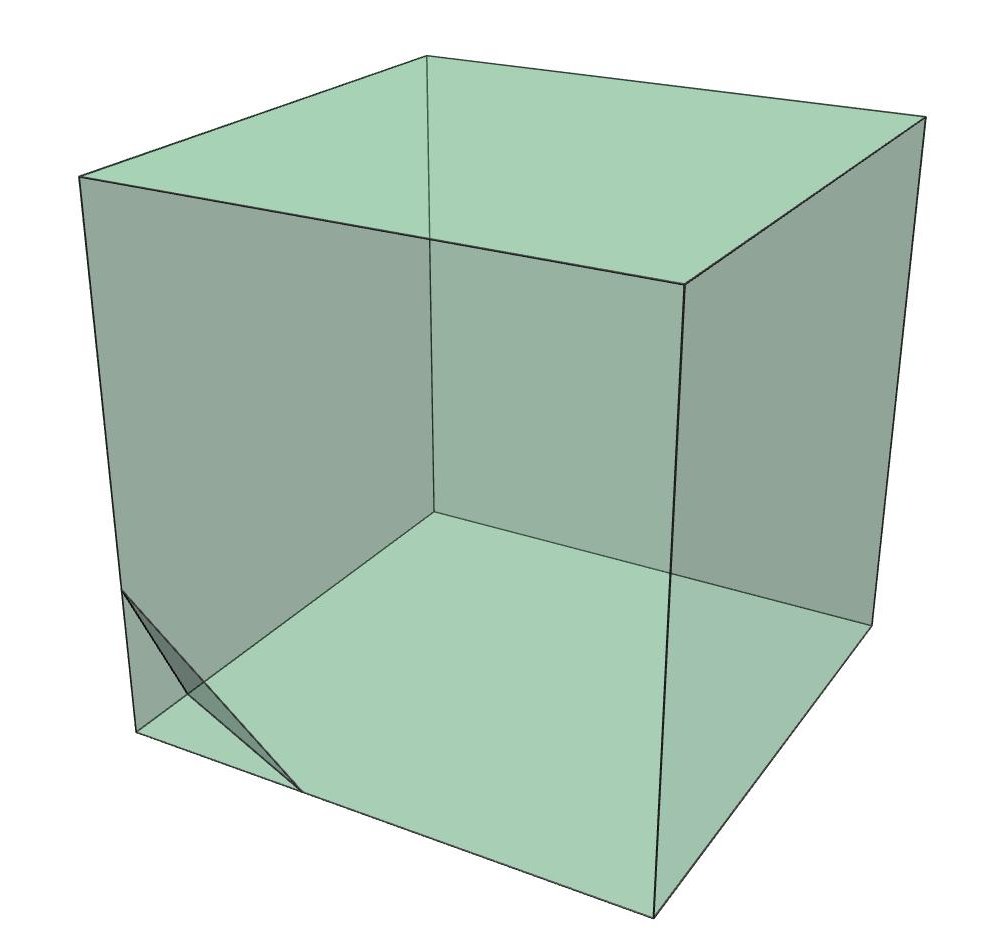}  \\
  \begin{tikzpicture}
   \draw (0,0) -- (1,0) -- (2,0) -- (3,0);
   \draw (0,0) -- (-1,0) node[left]{5};
   \draw (0,0) -- (-1,-0.3) node[left]{6};
   \draw (1,0) -- (-1,0.3) node[left]{4};
   \draw (2,0) -- (4,-0.3) node[right]{3};
   \draw (3,0) -- (4,0) node[right]{2};
   \draw (3,0) -- (4,0.3) node[right]{1};
   \fill[black] (3,0) circle (2pt) node[above]{$p_1$};
   \fill[black] (1,0) circle (2pt) node[below]{$p_0$};
  \end{tikzpicture} & 
  \begin{tikzpicture}
   \draw (0,0) -- (1,0) -- (2,0) -- (3,0);
   \draw (0,0) -- (-1,0) node[left]{5};
   \draw (0,0) -- (-1,-0.3) node[left]{6};
   \draw (1,0) -- (-1,0.3) node[left]{4};
   \draw (2,0) -- (4,-0.3) node[right]{3};
   \draw (3,0) -- (4,0) node[right]{2};
   \draw (3,0) -- (4,0.3) node[right]{1};
   \fill[black] (0,0) circle (2pt) node[below]{$p_0$};
   \fill[black] (3,0) circle (2pt) node[above]{$p_1$};
  \end{tikzpicture}  \\
  $10\beta + 4\alpha = 1$ & $5\gamma + 10\beta + 4\alpha = 1$\\
  & \\
  \hline
 \end{tabular}
 
 \caption{Different choices of vertices yield different cells of $\hw{2}(x,p)$.}\label{hurwitz_fig_choices}
\end{figure}

\section{Properties of Hurwitz cycles}\label{section_properties}

In the first two parts of this section we want to study whether tropical Hurwitz cycles are irreducible. For this purpose we will first prove that all (marked and unmarked) Hurwitz cycles are connected in codimension one. We will go on to show that for a generic choice of $p_j$ all marked cycles $\hwt{k}(x,p)$ are locally and globally  a multiple of an irreducible cycle. Finally we will see that $\hw{k}(x,p)$ is in general not irreducible.

\subsection{Connectedness in codimension one}\label{section_connected}

It is well known that $\mk{n}$ is connected in codimension one. In this particular case, the property has a very nice combinatorial description: maximal cones correspond to rational curves with $n-3$ bounded edges. A codimension one face of a maximal cone is attained by shrinking any of these edges to length 0, thus obtaining a single four-valent vertex. This vertex can then be \enquote{drawn apart} or \emph{resolved} in three different ways, thus moving into a maximal cone again. Saying that $\mk{n}$ is connected in codimension one means that we can transform any three-valent curve into another by alternatingly contracting edges and resolving four-valent vertices.

A similar correspondence holds for Hurwitz covers. An element of a maximal cone of $\hwt{k}(x,p) \subseteq \mk{N}(\R,x)$ can be considered as an $n$-marked rational curve $C$ with $N = n-2-k$ additional leaves attached to vertices of $C$. By abuse of notation, throughout this chapter we will also label these additional leaves by $p_0,\dots,p_{N-1}$. By the \emph{valence} of a vertex of an element of $\hwt{k}(x,p)$, we will mean the valence of the vertex in the underlying $n$-marked curve. 

For a generic choice of $p$, maximal cells of $\hwt{k}(x,p)$ will also correspond to curves with $n-3$ bounded edges and codimension one cells are obtained by shrinking an edge. Hence the problem of connectedness can be formulated in the same manner as for $\mk{n}$. However, the requirement that the contracted leaves be mapped to specific points excludes certain combinatorial \enquote{moves}, as we will shortly see.

Also note that the problem of connectedness does not really change if we allow non-generic points: the combinatorial problem remains essentially the same, we just allow some edge lengths to be 0. Hence we will assume throughout this section that $p_0 < p_1 < \dots < p_{N-1}$.

We will first show connectedness in the case $k = 1$. In this case the Hurwitz cycle is a tropical curve, so saying that $\hwt{1}(x,p)$ is connected in codimension one is the same as requiring that it is path-connected. So we will prove that for each two vertices $q,q'$ of $\hwt{1}(x,p)$ there exists a sequence of edges connecting them. 

We will prove this by induction on $n$, the length of $x$. For the case $n=5$ we will simply go through all possible cases explicitly. For $n > 5$, we will first show that any two covers of a special type, called \emph{chain covers}, are connected. Having shown this, we will then introduce a construction that allows us to connect any cover to a chain cover.

The general case is then an easy corollary, since we mark fewer vertices in higher-dimensional Hurwitz cycles, thus obtaining more degrees of freedom.

\begin{remark}\label{remark_bad_moves}
 Before we start, we want to discuss why this problem is so difficult. Since $\mk{N}$ is connected in codimension one, one would expect to be able to move from one combinatorial type to another without problems. However, the intermediate types need not be valid covers: a vertex of $\hwt{1}(x,p)$ can be considered as a point in a codimension one cone of $\mk{n}$, i.e.\ a curve with one four-valent vertex and only trivalent vertices besides, with an additional marked end attached to every vertex. Moving along an edge of $\hwt{1}(x,p)$ means moving an edge or leaf of that codimension one type along a bounded edge. However, this cannot be done in an arbitrary manner, since not all of these movements will produce valid covers (see figure \ref{figure_bad_move_ex} for an example). Note that the $p_j$ already fix the length of all bounded edges of a vertex curve in $\hwt{1}(x,p)$ uniquely. So, we will usually identify each vertex of $\hwt{1}(x,p)$ with the combinatorial type of the corresponding curve. 

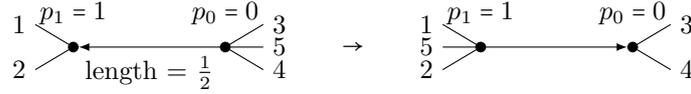
\begin{figure}[ht]
 \centering
  \begin{tikzpicture}
   \matrix[column sep = 5mm]{
	\draw[<-, shorten <=2pt] (0,0)  node[above = 5pt]{$p_1 = 1$} -- (2,0) node[above = 5pt]{$p_0 = 0$};
	\fill[black] (0,0) circle (2pt);
	\fill[black] (2,0) circle (2pt);
	\draw (1,0) node[below]{length = $\frac{1}{2}$};
	\draw (0,0) -- (-0.5,0.3) node[left]{1};
	\draw (0,0) -- (-0.5,-0.3) node[left]{2};
	\draw (2,0) -- (2.5, 0.3) node[right]{3};
	\draw (2,0) -- (2.5,0) node[right]{5};
	 \draw (2,0) -- (2.5,-0.3) node[right]{4}; &
	\draw (0,0) node{$\to$}; &
	\draw[->,shorten >=2pt] (0,0)  node[above = 5pt]{$p_1 = 1$} -- (2,0) node[above = 5pt]{$p_0 = 0$};
	\fill[black] (0,0) circle (2pt);
	\fill[black] (2,0) circle (2pt);
	\draw (0,0) -- (-0.5,0.3) node[left]{1};
	\draw (0,0) -- (-0.5,-0.3) node[left]{2};
	\draw (0,0) -- (-0.5,0) node[left]{5};
	\draw (2,0) -- (2.5, 0.3) node[right]{3};
	 \draw (2,0) -- (2.5,-0.3) node[right]{4};   
	\\
      };
  \end{tikzpicture}
  \caption{The curve on the left is a vertex of $\hwt{1}(1,1,1,1,-4)$. In $\curly{M}_{0,7}$ it corresponds to a ray spanning a cone with the curve on the right. However, the right curve is not an element of $\hwt{1}(1,1,1,1,-4)$ (for any edge length), since the edge direction is not compatible with the vertex ordering.}\label{figure_bad_move_ex}
 \end{figure}

Recall that the \emph{weight} or slope of an edge $e$ is $x_e := \abs{\sum_{i\in I} x_i}$, where $I$ is the split on $[n]$ induced by $e$. The \emph{orientation} of $e$ is chosen as in Example \ref{hurwitz_ex_cover}: $e$ \enquote{points towards $I$} if and only if $\sum_{i \in I} x_i > 0$.

Now, when moving some leaf along a bounded edge, that edge might change direction. But the direction of the edges is dictated by the order of the $p_i$, so this is not a valid move. One can easily see the following (see figure \ref{figure_invalid_moves} for an illustration): moving an edge/leaf $i$ to the other side of a bounded edge $e$ changes the direction of that edge if and only if one of them is incoming and one outgoing (recall that we consider leaves as incoming if they have negative weight) and $\abs{x_i} > x_e$. Note that, even if the direction of an edge does not change, moving an edge might be illegal (see the last diagram in figure \ref{figure_invalid_moves}), if the resulting edge configuration does not agree with the order on the $p_j$.
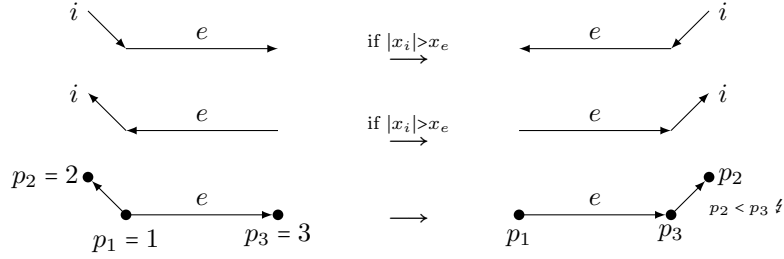
\begin{figure}[ht]
 \centering
  \begin{tikzpicture}
   \matrix[column sep = 5mm]{
	\draw[-latex] (0,0) -- (1,0) node[above]{$e$} -- (2,0);
	\draw[latex-] (0,0) -- (-0.5,0.5) node[left]{$i$}; &
	\draw (0,0) node{$\stackrel{\textnormal{if } \abs{x_i} > x_e}{\longrightarrow}$};&
	\draw[latex-] (0,0) -- (1,0) node[above]{$e$} -- (2,0);
	\draw[latex-] (2,0) -- (2.5,0.5) node[right]{$i$};\\	
	\draw[latex-] (0,0) -- (1,0) node[above]{$e$} -- (2,0);
	\draw[-latex] (0,0) -- (-0.5,0.5) node[left]{$i$}; &
	\draw (0,0) node{$\stackrel{\textnormal{if } \abs{x_i} > x_e}{\longrightarrow}$};&
	\draw[-latex] (0,0) -- (1,0) node[above]{$e$} -- (2,0);
	\draw[-latex] (2,0) -- (2.5,0.5) node[right]{$i$};\\
	\draw[-latex,shorten >= 2pt] (0,0) node[below=2pt]{$p_1 = 1$} -- (1,0) node[above]{$e$} -- (2,0) node[below]{$p_3 = 3$}; 
	\fill[black] (-0.5,0.5) circle (2pt);
	\fill[black] (0,0) circle (2pt);
	\fill[black] (2,0) circle (2pt);
	\draw[-latex,shorten >= 2pt] (0,0) -- (-0.5,0.5) node[left]{$p_2 = 2$}; &
	\draw (0,0) node{$\stackrel{}{\longrightarrow}$}; &
	\draw[-latex,shorten >= 2pt] (0,0) node[below=2pt]{$p_1$} -- (1,0) node[above]{$e$} -- (2,0) node[below]{$p_3$}; 
	\draw[-latex, shorten >= 2pt] (2,0) -- (2.5,0.5) node[right]{$p_2$};
	\fill[black] (2.5,0.5) circle (2pt);
	\fill[black] (0,0) circle (2pt);
	\fill[black] (2,0) circle (2pt);
	\draw (2.2,0.1) node[right = 5pt]{\tiny $p_2 < p_3 \;\lightning $};
	\\
      };
  \end{tikzpicture}
  \caption{Invalid moves on a Hurwitz cover: in the first two cases, when moving the leaf/edge  $i$ along the bounded edge $e$, the direction of $e$ changes. In the third case the edge direction of $e$ remains the same, but the direction is not compatible with the order of the $p_i$.}\label{figure_invalid_moves}
 \end{figure}
\end{remark}

\begin{defn}
A \emph{vertex type cover} is any cover corresponding to a vertex of $\hwt{1}(x,p)$. 
\end{defn}

\begin{lemma}\label{lemma_connected_five}
 For $n=5$, the cycle $\hwt{1}(x,p)$ is connected in codimension one for any $p$ and $x$.
\begin{proof}
 Let $q,q'$ be two vertices of $\hwt{1}(x,p)$ and $C,C'$ the corresponding rational curves. Both curves consist of a single bounded edge connecting contracted ends $p_0 < p_1$ with three leaves on one side and two on the other. We distinguish different cases, depending on how many leaves have to switch sides to go from $C$ to $C'$.

 Assume first that both curves only differ by the placement of one leaf, i.e.\ we want to move one leaf $i$ from the four-valent vertex in $C$ to the other side of the bounded edge. We can assume without restriction that the four-valent vertex in $C$ is at $p_0$. Assume that moving $i$ to the other side is an invalid move. Then the direction of the bounded edge would be inverted in $C'$, which is a contradiction to the fact that $p_0 < p_1$.

Now assume that both curves differ by an exchange of two leaves. Again we assume that the four-valent vertex in $C$ (and hence also in $C'$) is at $p_0$. Denote the leaves in $C$ at $p_0$ by $i,a,b$ and the remaining two at $p_1$ by $j,c$ and assume that $C'$ is obtained by exchanging $i$ and $j$. If we can move either $i$ in $C$ or $j$ in $C'$, then we are in the case where only one leaf needs to be moved, which we already studied. So assume that $i$ and $j$ cannot be moved in $C$ and $C'$, respectively. By remark \ref{remark_bad_moves}, this means that $x_i < - x_e < 0$, where $x_e$ is the weight of the bounded edge in $C$. Furthermore, $x_i + x_a + x_b = - x_e$, so $x_a + x_b > 0$. We assume without restriction that $x_a > 0$. Hence we can move $a$ along the bounded edge to obtain a valid cover $C_1$, whose four-valent vertex is at $p_1$. Since we assumed that we cannot move $j$ in $C'$, we must have $x_j < 0$ (it must be an incoming edge). This implies that we can move it to the left in $C_1$ to obtain a cover $C_2$. We now have $i,j,b$ at $p_0$ and $c,a$ at $p_1$. We want to show that we can move $i$ to the other side.  Assume this is not possible. Then $- x_i > x_e'$, where $x_e'$ is the weight of the bounded edge in $C_2$. But $x_e' = - x_i - x_j - x_b$. This implies $0 > - x_j - x_b$. Again, since $j$ cannot be moved in $C'$ we have $- x_j > x_a + x_b$. Finally, we obtain that $0 > x_a + x_b - x_b = x_a > 0$, which is a contradiction. Thus we can move $i$ to the right side to obtain a cover $C_3$. This cover now only differs from $C'$ by the placement of leaf $a$, so we are again in the first case (see figure \ref{figure_exchange_leaves} for an illustration).
\begin{figure}[ht]
 \centering
  \begin{tikzpicture}
   \matrix[column sep = 1mm]{
	\draw[->, shorten >= 2pt] (0,0) node[above=2pt]{$p_0$} -- (0.5,0) node[below=4pt]{$C$} -- (1,0) node[above=2pt]{$p_1$};
	\draw (0,0) -- (-0.5,0.3) node[left]{$i$};
	\draw (0,0) -- (-0.5,0) node[left]{$a$};
	\draw (0,0) -- (-0.5,-0.3) node[left]{$b$};
	\draw (1,0) -- (1.5,0.3) node[right]{$j$};
	\draw (1,0) -- (1.5,-0.3) node[right]{$c$};
	\fill[black] (0,0) circle (2pt); 
	\fill[black] (1,0) circle (2pt);&
	\draw (0,0) node{$\leadsto$};&
	\draw[->, shorten >= 2pt] (0,0) node[above=2pt]{$p_0$} -- (0.5,0) node[below=4pt]{$C_1$} -- (1,0) node[above=2pt]{$p_1$};
	\draw (0,0) -- (-0.5,0.3) node[left]{$i$};
	\draw (0,0) -- (-0.5,-0.3) node[left]{$b$};
	\draw (1,0) -- (1.5,0.3) node[right]{$j$};
	\draw[color = DarkRed] (1,0) -- (1.5,0) node[right]{$a$};
	\draw (1,0) -- (1.5,-0.3) node[right]{$c$};
	\fill[black] (0,0) circle (2pt); 
	\fill[black] (1,0) circle (2pt);&
	\draw (0,0) node{$\leadsto$};&
	\draw[->, shorten >= 2pt] (0,0) node[above=2pt]{$p_0$} -- (0.5,0) node[below=4pt]{$C_2$} -- (1,0) node[above=2pt]{$p_1$};
	\draw (0,0) -- (-0.5,0.3) node[left]{$i$};
	\draw[color = DarkRed] (0,0) -- (-0.5,0) node[left]{$j$};
	\draw (0,0) -- (-0.5,-0.3) node[left]{$b$};
	\draw (1,0) -- (1.5,0.3) node[right]{$a$};
	\draw (1,0) -- (1.5,-0.3) node[right]{$c$};
	\fill[black] (0,0) circle (2pt); 
	\fill[black] (1,0) circle (2pt);\\
	& \draw (0,0) node{$\leadsto$};&
	\draw[->, shorten >= 2pt] (0,0) node[above=2pt]{$p_0$} -- (0.5,0) node[below=4pt]{$C_3$} -- (1,0) node[above=2pt]{$p_1$};
	\draw (0,0) -- (-0.5,0.3) node[left]{$j$};
	\draw (0,0) -- (-0.5,-0.3) node[left]{$b$};
	\draw (1,0) -- (1.5,0.3) node[right]{$a$};
	\draw[color = DarkRed] (1,0) -- (1.5,0) node[right]{$i$};
	\draw (1,0) -- (1.5,-0.3) node[right]{$c$};
	\fill[black] (0,0) circle (2pt); 
	\fill[black] (1,0) circle (2pt);&
	\draw (0,0) node{$\leadsto$};&
	\draw[->, shorten >= 2pt] (0,0) node[above=2pt]{$p_0$} -- (0.5,0) node[below=4pt]{$C'$} -- (1,0) node[above=2pt]{$p_1$};
	\draw (0,0) -- (-0.5,0.3) node[left]{$j$};
	\draw[color = DarkRed] (0,0) -- (-0.5,0) node[left]{$a$};
	\draw (0,0) -- (-0.5,-0.3) node[left]{$b$};
	\draw (1,0) -- (1.5,0.3) node[right]{$i$};
	\draw (1,0) -- (1.5,-0.3) node[right]{$c$};
	\fill[black] (0,0) circle (2pt); 
	\fill[black] (1,0) circle (2pt);\\
      };
  \end{tikzpicture}
  \caption{Connecting two curves differing by an exchange of leaves. The leaf we moved in each step is marked by a red line.}\label{figure_exchange_leaves}
\end{figure}
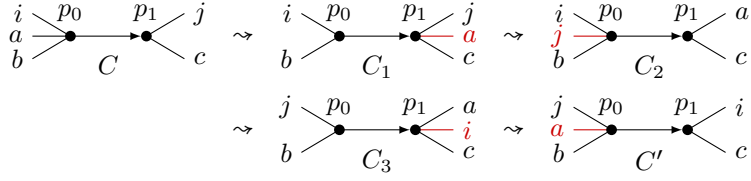

Now assume we have to move three leaves (see figure \ref{figure_move_three}). That means we have to exchange two leaves $i,j$ from the four-valent vertex in $C$ (again assume it is at $p_0$) for one leaf $k$ at $p_1$. Assume we cannot move $i$ in $C$. In particular, $x_i < 0$. But that means we can move $i$ in $C'$ to obtain a cover $C_1$. This cover differs from $C$ by the exchange of $j$ and $k$, so we already know they are connected.

\begin{figure}[ht]
 \centering
  \begin{tikzpicture}
   \matrix[column sep = 10mm]{
	\draw[->, shorten >= 2pt] (0,0) node[above=2pt]{$p_0$} -- (0.5,0) node[below=4pt]{$C$} -- (1,0) node[above=2pt]{$p_1$};
	\draw (0,0) -- (-0.5,0.3) node[left]{$i$};
	\draw (0,0) -- (-0.5,0) node[left]{$j$};
	\draw (0,0) -- (-0.5,-0.3) node[left]{$a$};
	\draw (1,0) -- (1.5,0.3) node[right]{$k$};
	\draw (1,0) -- (1.5,-0.3) node[right]{$b$};
	\fill[black] (0,0) circle (2pt);
	\fill[black] (1,0) circle (2pt);
	&
	\draw[->,shorten >= 2pt] (0,0) node[above=2pt]{$p_0$} -- (0.5,0) node[below=4pt]{$C'$} -- (1,0) node[above=2pt]{$p_1$};
	\draw (0,0) -- (-0.5,0.3) node[left]{$k$};
	\draw (0,0) -- (-0.5,-0.3) node[left]{$a$};
	\draw (1,0) -- (1.5,0.3) node[right]{$i$};
	\draw (1,0) -- (1.5,0) node[right]{$j$};
	\draw (1,0) -- (1.5,-0.3) node[right]{$b$};
	\fill[black] (0,0) circle (2pt);
	\fill[black] (1,0) circle (2pt);\\
      };
  \end{tikzpicture}
  \caption{Two vertex types differing by a movement of three leaves. Depending on the direction of $i$, we can move it either in $C$ or in $C'$.}\label{figure_move_three}
\end{figure}
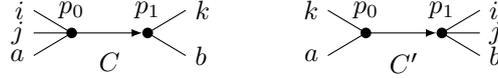

Finally, assume that four leaves have to switch sides, i.e.\ we exchange two leaves $i,j$ at the four-valent at $p_0$ for the two leaves $k,l$ at $p_1$. Assume we can move neither $i$ nor $j$. This means that $x_i,x_j < 0$. But then $x_i + x_j < 0$ as well, so the edge direction would be inverted in $C'$, which is a contradiction. Hence we can move $i$ or $j$ and reduced the problem to the case where only three leaves need to be moved.

It is easy to see that these are all possible cases. In particular, it is impossible to let all five leaves switch sides, since this would automatically invert the direction of the bounded edge.
\end{proof}
\end{lemma}

As mentioned above, we want to show that for $n > 5$ we can connect each vertex type to a vertex corresponding to a \emph{standard cover}. Let us define this:

\begin{defn}\label{connect_defn_std}
 Let $x \in \curly{H}_n$. We define an order $<_x$ on $[n]$ by:
$$i <_x j :\iff x_i < x_j \textnormal{ or } (x_i = x_j \textnormal{ and } i < j). $$
A \emph{chain cover for $x$} is a vertex type cover with the additional property that the vertex marked with $p_i$ is connected to the vertex marked with $p_j$, if and only if $\abs{i - j} = 1$ (i.e.\ the $p_j$ are arranged as a single chain in order of their size). Fix an $s \in \{0,\dots,n-4\}$. The \emph{standard cover for $x$ at $p_s$} is the unique chain cover, where the leaves are attached to the $p_j$ according to their size (defined by $<_x$) and $p_s$ is at the four-valent vertex. More precisely: if leaf $i$ is attached to $p_k$ and leaf $j$ is attached to $p_l$, then $i <_x j \iff p_k < p_l$ (See figure \ref{figure_standard_cover} for an example of this construction).
\begin{figure}[ht]
 \centering
  \begin{tikzpicture}
    \draw[->, shorten >= 2pt] (0,0) node[above=2pt]{$p_0$} -- (1,0) node[above=2pt]{$p_1$};
    \draw[->, shorten >= 2pt] (1,0) -- (2,0) node[above=2pt]{$p_2$}; 
    \draw[->, shorten >= 2pt] (2,0) -- (3,0) node[above=2pt]{$p_3$};
    \draw (3,0) -- (3.5,0.3) node[right]{$x_1 = 3$};
    \draw (3,0) -- (3.5,0) node[right]{$x_2 = 2$};
    \draw (3,0) -- (3.5,-0.3) node[right]{$x_7 = 1$};
    \draw (2,0) -- (2.3,-0.5) node[below]{$x_3 = 1$};
    \draw (1,0) -- (0.7,-0.5) node[below]{$x_5 = -1$};
    \draw (0,0) -- (-0.5,0.3) node[left]{$x_4 = -3$};
    \draw (0,0) -- (-0.5,-0.3) node[left]{$x_6 = -3$};
    \fill[black] (0,0) circle (2pt);
    \fill[black] (1,0) circle (2pt);
    \fill[black] (2,0) circle (2pt);
    \fill[black] (3,0) circle (2pt);
  \end{tikzpicture}
  \caption{The standard cover for $x= (3,2,1,-3,-1,-3,1)$ at $p_3$}\label{figure_standard_cover}
\end{figure}
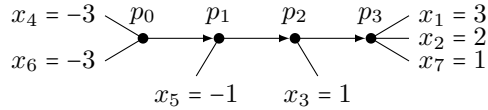
\end{defn}

\begin{lemma}
 Each standard cover is a valid Hurwitz cover.
\begin{proof}
 We have to show that the edge connecting $p_j$ and $p_{j+1}$ points towards $p_{j+1}$ for all $j$. Note that the weight and direction of an edge only depend on the split defined by it.

We will say that a leaf \emph{lies behind} $p_k$, if it is attached to some $p_{k'}$, $k' \geq k$. Denote the leaves lying behind $p_{j+1}$ by $i_1,\dots,i_l$. Their weights are by construction larger than or equal to all weights of remaining leaves. Considering that the sum over all leaves is 0, this implies that $\sum_{s = 1}^l x_{i_s} > 0$ (if it was 0, then all $x_i$ would have to be 0). Hence the bounded edge points towards $p_{j+1}$.
\end{proof}
\end{lemma}

We will also need another construction in our proofs:

\begin{defn}
 Let $C$ be a vertex type cover and $e$ any bounded edge in $C$ connecting the contracted ends $p$ and $q$. Removing $e$, we obtain two path-connected components. For any contracted end $r$, we write $C_e(r)$ for the component containing $r$.

 Now assume $C_e(p)$ contains the four-valent vertex and at least one other bounded edge. The \emph{split cover at $e$} is a cover $C'$ obtained in the following way: remove the edge $e$ and keep only $C_e(p)$. Then attach a leaf to $p$ whose weight is the original weight of $e$ (or its negative, if $e$ pointed towards $p$). This is obviously a vertex type cover for some $x' = (x_1',\dots,x_m')$, where $m < n$  (see figure \ref{figure_split} for an example). We denote the leaf replacing $e$ by $l_e$ and call it the \emph{splitting leaf}.
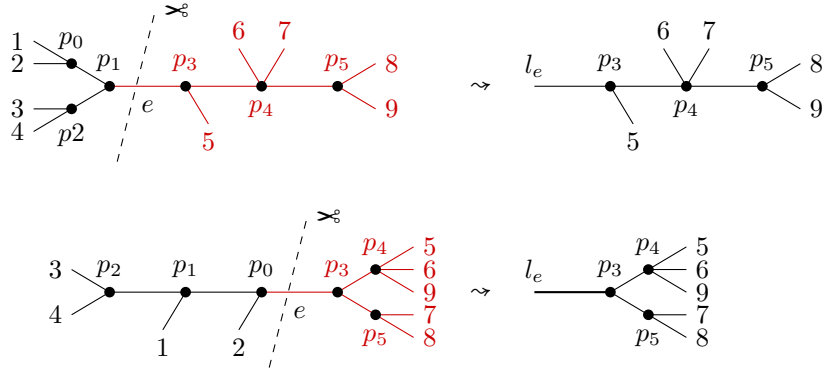
\begin{figure}[ht]
 \centering
  \begin{tikzpicture}
   \matrix[column sep = 2mm,row sep = 5mm]{
	\draw (0,0) node[above=2pt]{$p_1$} -- (-0.5,0.3) node[above=2pt]{$p_0$};
	\draw (0,0) -- (-0.5,-0.3) node[below=2pt]{$p2$};
	\draw (-0.5,0.3) -- (-1,0.3) node[left]{$2$};
	\draw (-0.5,0.3) -- (-1,0.6) node[left]{$1$};
	\draw (-0.5,-0.3) -- (-1,-0.3) node[left]{$3$};
	\draw (-0.5,-0.3) -- (-1,-0.6) node[left]{$4$};
	\draw[color = DarkRed] (0,0) -- (1,0) node[above=2pt]{$p_3$};
	\draw (0.5,0) node[below=2pt]{$e$};
	\draw[color = DarkRed] (1,0) -- (1.3,-0.5) node[below]{$5$};
	\draw[color = DarkRed] (1,0) -- (2,0) node[below=2pt]{$p_4$};
	\draw[color = DarkRed] (2,0) -- (1.7,0.5) node[above]{$6$};
	\draw[color = DarkRed] (2,0) -- (2.3,0.5) node[above]{$7$};
	\draw[color = DarkRed] (2,0) -- (3,0) node[above=2pt]{$p_5$};
	\draw[color = DarkRed] (3,0) -- (3.5,0.3) node[right]{$8$};
	\draw[color = DarkRed] (3,0) -- (3.5,-0.3) node[right]{$9$};
	\fill[black] (0,0) circle (2pt);
	\fill[black] (-0.5,0.3) circle (2pt);
	\fill[black] (-0.5,-0.3) circle (2pt);
	\fill[black] (1,0) circle (2pt);
	\fill[black] (2,0) circle (2pt);
	\fill[black] (3,0) circle (2pt);
	\draw[dashed] (0.1,-1) -- (0.6,1) node[right]{\Large\Leftscissors};
	&
	\draw (0,0) node{$\leadsto$};&
	\draw (0,0) node[above]{$l_e$} -- (1,0) node[above=2pt]{$p_3$};
	\draw (1,0) -- (1.3,-0.5) node[below]{$5$};
	\draw (1,0) -- (2,0) node[below=2pt]{$p_4$};
	\draw (2,0) -- (1.7,0.5) node[above]{$6$};
	\draw (2,0) -- (2.3,0.5) node[above]{$7$};
	\draw (2,0) -- (3,0) node[above=2pt]{$p_5$};
	\draw (3,0) -- (3.5,0.3) node[right]{$8$};
	\draw (3,0) -- (3.5,-0.3) node[right]{$9$};
	\fill[black] (1,0) circle (2pt);
	\fill[black] (2,0) circle (2pt);
	\fill[black] (3,0) circle (2pt);
	\\
	\draw (0,0) node[above=2pt]{$p_2$} -- (1,0) node[above = 2pt]{$p_1$};
	\draw (1,0) -- (2,0) node[above=2pt]{$p_0$};
	\draw (0,0) -- (-0.5,0.3) node[left]{$3$};
	\draw (0,0) -- (-0.5,-0.3) node[left]{$4$};
	\draw (1,0) -- (0.7,-0.5) node[below]{$1$};
	\draw (2,0) -- (1.7,-0.5) node[below]{$2$};
	\draw[color = DarkRed] (2,0) -- (3,0) node[above=2pt]{$p_3$};
	\draw (2.5,0) node[below=2pt]{$e$};
	\draw[color = DarkRed] (3,0) -- (3.5,0.3) node[above=2pt]{$p_4$};
	\draw[color = DarkRed] (3,0) -- (3.5,-0.3) node[below=2pt]{$p_5$};
	\draw[color = DarkRed] (3.5,0.3) -- (4,0.6) node[right]{$5$};
	\draw[color = DarkRed] (3.5,0.3) -- (4,0.3) node[right]{$6$};
	\draw[color = DarkRed] (3.5,0.3) -- (4,0) node[right]{$9$};
	\draw[color = DarkRed] (3.5,-0.3) -- (4,-0.3) node[right]{$7$};
	\draw[color = DarkRed] (3.5,-0.3) -- (4,-0.6) node[right]{$8$};
	\fill[black] (0,0) circle (2pt);
	\fill[black] (1,0) circle (2pt);
	\fill[black] (2,0) circle (2pt);
	\fill[black] (3,0) circle (2pt);
	\fill[black] (3.5,0.3) circle (2pt);
	\fill[black] (3.5,-0.3) circle (2pt);
	\draw[dashed] (2.1,-1) -- (2.6,1) node[right]{\Large\Leftscissors};
	&
	\draw (0,0) node{$\leadsto$};&
	\draw[thick] (0,0) node[above]{$l_e$} -- (1,0) node[above=2pt]{$p_3$};
	\draw (1,0) -- (1.5,0.3) node[above=2pt]{$p_4$};
	\draw (1,0) -- (1.5,-0.3) node[below=2pt]{$p_5$};
	\draw (1.5,0.3) -- (2,0.6) node[right]{$5$};
	\draw (1.5,0.3) -- (2,0.3) node[right]{$6$};
	\draw (1.5,0.3) -- (2,0) node[right]{$9$};
	\draw (1.5,-0.3) -- (2,-0.3) node[right]{$7$};
	\draw (1.5,-0.3) -- (2,-0.6) node[right]{$8$};
	\fill[black] (1,0) circle (2pt);
	\fill[black] (1.5,0.3) circle (2pt);
	\fill[black] (1.5,-0.3) circle (2pt);
	\\
      };
  \end{tikzpicture}
  \caption{Two Hurwitz covers for $n = 9$. In each case the split cover at the edge marked by $e$ is a cover for $n=6$ (the labels at the leaves are just indices in this case, not weights).}\label{figure_split}
\end{figure}
\end{defn}

We now want to see that all chain covers are connected:

\begin{lemma}\label{lemma_chain_connected}
 Let $x \in \curly{H}_n$ and let $p_0,\dots,p_{n-4} \in \R$ with $p_j \leq p_{j+1}$ for all $j$. Then all chain covers for $x$ are connected to each other.
\begin{proof}
 We will show that all chain covers are connected to a standard cover at some $p_s$. We prove this by induction on $n$. For $n = 5$, all covers are chain covers and our claim follows from lemma \ref{lemma_connected_five}. 

So let $n > 5$ and $C$ be any chain cover. We can assume without restriction that the vertices at $p_0$ and $p_{n-4}$ are trivalent (if they are not, one can easily see that at least one leaf can be moved away). Take any bounded edge $e$ connecting some $p_j$ and $p_{j+1}$. Suppose there is a leaf $k$ at $p_j$ and a leaf $l$ at $p_{j+1}$, such that $k >_x l$. This means that exchanging $k$ and $l$ still gives a valid cover. We can assume without restriction that $j > 0$, i.e.\ $e$ is not the first edge (if $j = 0$, we can use a similar argument using a split cover at the edge connecting $p_{n-5}$ and $p_{n-4}$ ). 

Let $C'$ be the split cover at the edge connecting $p_0$ and $p_1$. This is a cover on $n-1$ leaves. By induction we know that $C'$ is connected to the cover which only differs from $C'$ by exchanging $k$ and $l$. Let $C''$ be any vertex type cover occurring along that path. Since $p_0$ is smaller than all $p_j$, we can lift $C''$ to a cover on $n$ leaves: simply re-attach the splitting leaf to $p_0$. (see figure \ref{figure_branch_sort} for an illustration of the split-and-lift construction in a different case).

Hence we obtain a path between $C$ and and the cover $\tilde{C}$, where $k$ and $l$ have been exchanged. We can apply this procedure iteratively to sort all leaves to obtain a standard cover at some $p_s$. 

Finally, note that all standard covers are connected: one can always move the smallest leaf at the four-valent vertex to the left (except of course at $p_0$) and the largest leaf to the right. This way the four-valent vertex can be placed at any contracted end.
\end{proof}
\end{lemma}

\begin{lemma}
 Let $x \in \curly{H}_n$. Then $\hwt{1}(x,p)$ is connected in codimension one.
\begin{proof}
 We prove this by induction on $n$. The case $n = 5$ was already covered in lemma \ref{lemma_connected_five}. Also note that for $n=4$ the Hurwitz cycle $\hwt{1}(x,p)$ is by definition equal to a Psi class and hence a fan curve.

So assume $n > 5$ and let $q$ be a vertex of $\hwt{1}(x,p)$ with corresponding rational curve $C$. We want to show that it is connected to the standard cover on $p_0$. First, we prove the following technical statement:

\emph{1) Let $e$ be a bounded edge connecting $p_0$ and some $p_j$, such that $C_e(p_j)$ contains the four-valent vertex. Let $C'$ be the split cover at $e$ with degree $x' = (x_1',\dots,x_m')$. Let $P = \{p_1',\dots,p_m'\}$ be the set of contracted ends in $C'$ and assume $p_1' < \dots < p_m'$. Then $C$ is connected to the cover $C''$, obtained in the following way: first, remove all leaves and contracted ends contained in $C'$ from $C$ together with any bounded edges that are attached to them. Then attach all $p \in P$ as an ordered chain to $p_0$, i.e.\ $p_1'$ to $p_0$, $p_2'$ to $p_1', \dots$ etc. Assume the leaves in $C'$ have weights $x_{i_1} \leq \dots \leq x_{i_{m-1}}$. Attach leaf $i_1$ to $p_0$, $i_2$ to $p_1'$ and so on (see figure \ref{figure_branch_sort}).}

\begin{figure}[ht]
 \centering
  \begin{tikzpicture}
   \matrix[column sep = 2mm, row sep = 5mm]{
      \draw[<-,shorten <= 2pt] (0,0) node[above = 2pt]{$p_1$} node[left = 20pt]{$C := $}-- (1,0) node[above=2pt]{$p_0$};
      \draw[->, shorten >= 2pt, color = red] (1,0) -- (1.5,0) node[above]{$e$} -- (2,0) node[above=2pt,color = black]{$p_2$};
      \draw[->,shorten >= 2pt] (2,0) -- (2.5,0.3) node[above=2pt]{$p_3$};
      \draw[->, shorten >= 2pt] (2,0) -- (2.5,-0.3) node[below=3pt]{$p_4$};
      \draw (0,0) -- (-0.5,0.3) node[left]{1};
      \draw (0,0) -- (-0.5,-0.3) node[left]{1};
      \draw (1,0) -- (0.7,-0.5) node[left]{-3};
      \draw (2,0) -- (1.7,-0.5) node[left]{-3};
      \draw (2.5,0.3) -- (3,0.6) node[right]{4};
      \draw (2.5,0.3) -- (3,0.3) node[right]{-2};
      \draw (2.5,-0.3) -- (3,-0.3) node[right]{4};
      \draw (2.5,-0.3) -- (3,-0.6) node[right]{-2}; 
      \fill[black] (0,0) circle (2pt);
      \fill[black] (1,0) circle (2pt);
      \fill[black] (2,0) circle (2pt);
      \fill[black] (2.5,0.3) circle (2pt);
      \fill[black] (2.5,-0.3) circle (2pt);
      & \draw (0,0) node{$\stackrel{(1)}{\longrightarrow}$}; &
      \draw[->,shorten >= 2pt] (0,0) node[above=2pt]{$p_2$} node[right = 50pt]{$=: C'$}-- (0.5,0.3) node[above=2pt]{$p_3$};
      \draw[->, shorten >= 2pt] (0,0) -- (0.5,-0.3) node[below=3pt]{$p_4$};
      \draw (0,0) -- (-0.5,-0.3) node[left]{-3};
      \draw[color = red] (0,0) -- (-0.5,0.3) node[left]{$l_e = -1$};
      \draw (0.5,0.3) -- (1,0.6) node[right]{4};
      \draw (0.5,0.3) -- (1,0.3) node[right]{-2};
      \draw (0.5,-0.3) -- (1,-0.3) node[right]{4};
      \draw (0.5,-0.3) -- (1,-0.6) node[right]{-2}; 
      \fill[black] (0,0) circle (2pt);
      \fill[black] (0.5,0.3) circle (2pt);
      \fill[black] (0.5,-0.3) circle (2pt);
      \\      
      & \draw (0,0) node{$\stackrel{(2)}{\longrightarrow}$}; & 
      \draw[->, shorten >= 2pt] (0,0) node[above=2pt]{$p_2$} -- (1,0) node[above=2pt]{$p_3$};
      \draw[->, shorten >= 2pt] (1,0) -- (2,0) node[above=2pt]{$p_4$};
      \draw (0,0) -- (-0.5,0.3) node[left]{$-2$};
      \draw (0,0) -- (-0.5,-0.3) node[left]{$-3$};
      \draw (1,0) -- (0.7,-0.5) node[left]{$-2$};
      \draw[color = red] (1,0) -- (1.3,-0.5) node[right]{$l_e = -1$};
      \draw (2,0) -- (2.5,0.3) node[right]{$4$};
      \draw (2,0) -- (2.5,-0.3) node[right]{$4$};
      \fill[black] (0,0) circle (2pt);
      \fill[black] (1,0) circle (2pt);
      \fill[black] (2,0) circle (2pt);
      \\
      & \draw (0,0) node{$\stackrel{(3)}{\longrightarrow}$}; & 
      \draw[->, shorten >= 2pt] (0,0) node[above=2pt]{$p_2$} -- (1,0) node[above=2pt]{$p_3$};
      \draw[->, shorten >= 2pt] (1,0) -- (2,0) node[above=2pt]{$p_4$};
      \draw[color = red] (0,0) -- (-0.5,0.3) node[left]{$l_e = -1$};
      \draw (0,0) -- (-0.5,0) node[left]{$-2$};
      \draw (0,0) -- (-0.5,-0.3) node[left]{$-3$};
      \draw (1,0) -- (0.7,-0.5) node[left]{$-2$};
      \draw (2,0) -- (2.5,0.3) node[right]{$4$};
      \draw (2,0) -- (2.5,-0.3) node[right]{$4$};
      \fill[black] (0,0) circle (2pt);
      \fill[black] (1,0) circle (2pt);
      \fill[black] (2,0) circle (2pt);
      \\
      \draw[<-,shorten <= 2pt] (0,0) node[above = 2pt]{$p_1$} -- (1,0) node[above=2pt]{$p_0$};
      \draw[->, shorten >= 2pt, color = red] (1,0) -- (1.5,0) node[above]{$e$} -- (2,0) node[above=2pt,color = black]{$p_2$};
      \draw[->,shorten >= 2pt] (2,0) -- (3,0)node[above=2pt]{$p_3$};
      \draw[->,shorten >= 2pt] (3,0) -- (4,0)node[above=2pt]{$p_4$};
      \draw (0,0) -- (-0.5,0.3) node[left]{1};
      \draw (0,0) -- (-0.5,-0.3) node[left]{1};
      \draw (1,0) -- (0.7,-0.5) node[left]{-3};
      \draw (2,0) -- (1.7,-0.5) node[left]{-3};
      \draw (2,0) -- (2.3,-0.5) node[right]{-2};
      \draw (3,0) -- (3.3,-0.5) node[right]{-2};
      \draw (4,0) -- (4.5,0.3) node[right]{4};
      \draw (4,0) -- (4.5,-0.3) node[right]{4};
      \fill[black] (0,0) circle (2pt);
      \fill[black] (1,0) circle (2pt);
      \fill[black] (2,0) circle (2pt);
      \fill[black] (3,0) circle (2pt);
      \fill[black] (4,0) circle (2pt);
      & \draw (0,0) node{$\stackrel{(4)}{\longleftarrow}$};\\
      \draw[<-,shorten <= 2pt] (0,0) node[above = 2pt]{$p_1$} node[left = 20pt]{$C'':=$}-- (1,0) node[above=2pt]{$p_0$};
      \draw[->,shorten >= 2pt] (1,0) -- (2,0) node[above=2pt]{$p_2$};
      \draw[->, shorten >= 2pt] (2,0) -- (3,0)node[above=2pt]{$p_3$};
      \draw[->,shorten >= 2pt] (3,0) -- (4,0)node[above=2pt]{$p_4$};
      \draw (0,0) -- (-0.5,0.3) node[left]{1};
      \draw (0,0) -- (-0.5,-0.3) node[left]{1};
      \draw (1,0) -- (0.7,-0.5) node[left]{-3};
      \draw (2,0) -- (2.3,-0.5) node[right]{-2};
      \draw[color = DarkGreen] (1,0) -- (1.3,-0.5) node[right]{-3};
      \draw (3,0) -- (3.3,-0.5) node[right]{-2};
      \draw (4,0) -- (4.5,0.3) node[right]{4};
      \draw (4,0) -- (4.5,-0.3) node[right]{4};
      \fill[black] (0,0) circle (2pt);
      \fill[black] (1,0) circle (2pt);
      \fill[black] (2,0) circle (2pt);
      \fill[black] (3,0) circle (2pt);
      \fill[black] (4,0) circle (2pt);
      & \draw (0,0) node{$\stackrel{(5)}{\longleftarrow}$};      
      \\
      };
  \end{tikzpicture}
  \caption[justification=justified]{The branch sorting construction: 
   \\(1) Take the split cover $C'$ at $e$. \\(2) Move that split cover to a standard cover using induction. \\(3) Move the splitting leaf to the smallest $p_j$. \\(4) Consider the lift of this cover. \\(5) Move the smallest leaf at $p_1' = p_2$ to $p_0$ to obtain $C''$.
  }\label{figure_branch_sort}
\end{figure}
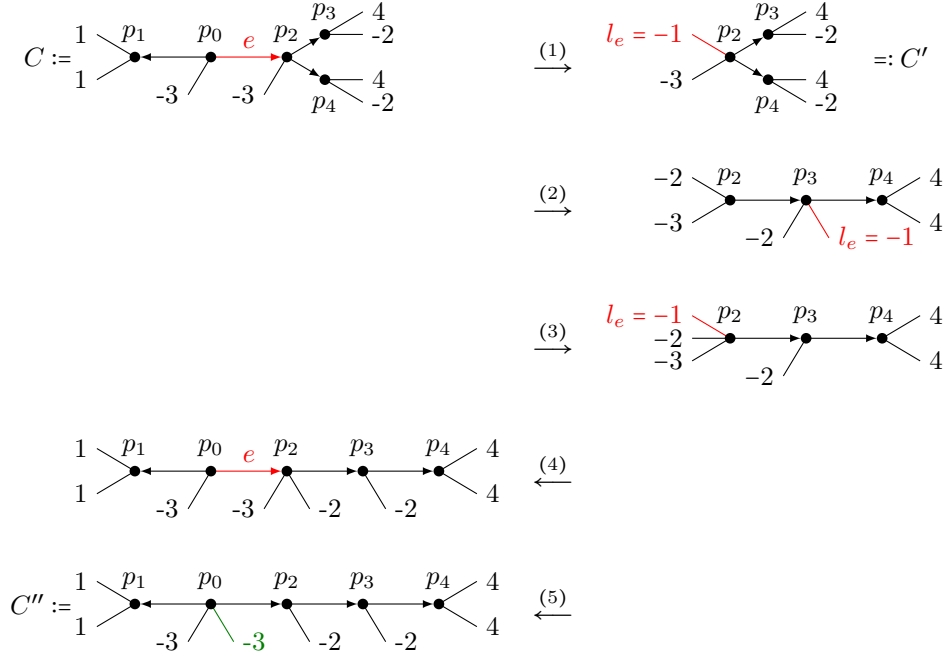

We know by induction that $C'$ is connected to the standard cover for $x'$ at any $p \in P$. Choose $p$, such that the standard cover at $p$ has the splitting leaf attached to the four-valent vertex. Since the splitting leaf has negative weight, we can move it to the smallest contracted end. This gives us a chain cover $C_2$ connected to $C'$. As in the proof of lemma \ref{lemma_chain_connected}, we can lift the connecting path to a path of covers with degree $x$ by attaching $p_0$ to the splitting leaf. Denote the lift of $C_2$ by $C_2'$. This cover has its four-valent vertex at $p_1'$. Denote by $k$ the smallest leaf at $p_1'$ with respect to $<_x$ and let $\omega$ be the weight of the edge connecting $p_0$ and $p_1'$. By definition $\omega = \sum_{i\in I} x_i$, where $I$ is the set of all leaves contained in $C'$. By construction, $k$ is the minimal element of $I$ with respect to $<_x$. Hence $\omega > k$ and we can move $k$ to $p_0$ to obtain $C''$. 

We can now use this to prove the following:

\emph{2) If $p_0$ has only one bounded edge attached to it, then $C$ is connected to the standard cover at $p_0$.}

We can assume without restriction that $p_0$ is not at the four-valent vertex (otherwise, we can move at least one leaf). We now apply the construction described in 1) to the single bounded edge at $p_0$. This gives us a chain cover for $x$, which by lemma \ref{lemma_chain_connected} is connected to the standard cover.

It remains to prove the following statement, which implies our theorem:

\emph{3) $C$ is always connected to a cover $C'$, in which $p_0$ has only one bounded edge attached to it.}

As any vertex is at most four-valent, $p_0$ can have at most four bounded edges attached to it. First, assume that only two bounded edges $e,e'$ are attached to $p_0$ and their other vertices are attached to contracted ends $p_e \leq p_{e'}$. If $p_0$ is four-valent, we can move $e'$ along $e$ to obtain a valid cover in which $p_0$ has a single bounded edge attached to it. If the four-valent vertex lies behind one of the edges, say $e$, we apply the construction of $1)$ to this edge. This way we obtain a cover in which $p_0$ is still attached to two bounded edges and is also four-valent.

Assume $p_0$ has three bounded edges $e,e',e''$ attached to it, connecting it to contracted ends $p_e \leq p_{e'} \leq p_{e''}$. With the same argument as in the case of two bounded edges, we can assume that the vertex at $p_0$ is four-valent. Now we can move $e'$ along $e$. Thus we obtain a cover where $p_0$ has only two bounded edges attached to it. A similar argument covers the case of four bounded edges (see also figure \ref{figure_p0_leaf} for an illustration in the case of two edges).

\begin{figure}[ht]
 \centering
  \begin{tikzpicture}
   \matrix[column sep = 3mm, row sep = 5mm]{
       \draw[->,shorten >= 2pt] (0,0) node[above = 2pt]{$p_0$} -- (1,0.3) node[above = 2pt]{$p_e$};
       \draw[->, shorten >= 2pt] (0,0) -- (1,-0.3) node[below = 2pt]{$p_{e'}$};
       \draw[->,dotted] (1,0.3) -- (2,0.6) node[above]{\tiny four-valent};
	\draw (0,0) -- (-0.5,0) node[left]{$i$};
	\fill[black] (0,0) circle (2pt);
	\fill[black] (1,0.3) circle (2pt);
	\fill[black] (1,-0.3) circle (2pt);
       & \draw (0,0) node{$\stackrel{1)}{\leadsto}$}; &
      \draw[->,shorten >= 2pt] (0,0) node[above = 2pt]{$p_0$} -- (1,0.3) node[above = 2pt]{$\tilde{p}_e$};
       \draw[->,shorten >= 2pt] (0,0) -- (1,-0.3) node[below = 2pt]{$p_{e'}$};
	\draw (0,0) -- (-0.5,0.3) node[left]{$i$};
% 	\draw (0,0) -- (-0.5,0);
	\draw (0,0) -- (-0.5,-0.3) node[left]{$j$};
	\fill[black] (0,0) circle (2pt);
	\fill[black] (1,0.3) circle (2pt);
	\fill[black] (1,-0.3) circle (2pt);
      & \draw (0,0) node{$\stackrel{\textnormal{if }p_{e'} > \tilde{p}_e}{\leadsto}$}; &
      \draw[->,shorten >= 2pt] (0,0) node[above = 2pt]{$p_0$} -- (1,0.3) node[above = 2pt]{$\tilde{p}_e$};
      \draw[->,shorten >= 2pt] (1,0.3) -- (2,0) node[above]{$p_{e'}$};
      \draw (0,0) -- (-0.5,0.3)node[left]{$i$};
% 	\draw (0,0) -- (-0.5,0);
	\draw (0,0) -- (-0.5,-0.3)node[left]{$j$};
	\fill[black] (0,0) circle (2pt);
	\fill[black] (1,0.3) circle (2pt);
	\fill[black] (2,0) circle (2pt);	
	\\
      };
  \end{tikzpicture}
  \caption{How to reduce the number of bounded edges at $p_0$: first move the four-valent vertex to $p_0$ using the construction from 1). Then move one bounded edge along another according to the size of the $p_e$.}\label{figure_p0_leaf}
\end{figure}
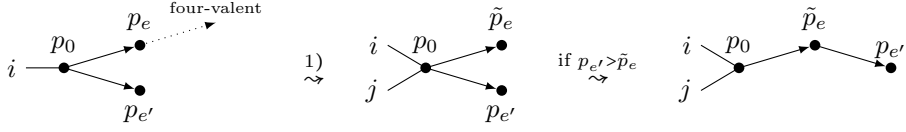

\end{proof}
\end{lemma}

\begin{theorem}\label{main_thm_connected}
 For all $k$, $p$ and $x$, the cycles $\hwt{k}(x,p)$ and $\hw{k}(x,p)$ are connected in codimension one.
\begin{proof}
 Note that it suffices to show the statement for $\hwt{k}(x,p)$, since $\hw{k}(x,p) = \ft(\hwt{k}(x,p))$ and connectedness in codimension one is independent of the chosen polyhedral structure.

The general idea of the proof is that for larger $k$ we mark fewer vertices with contracted ends and thus have more degrees of freedom to \enquote{move around}, so we can apply induction on $k$.

Similar to definition \ref{connect_defn_std} we define a \emph{sorted maximal cover} for $x$ on $S$, with $S \subseteq [n-2], \abs{S} = n-2-k$: we obtain a trivalent curve by attaching the leaves to a chain of $n-3$ bounded edges sorted according to the ordering $<_x$. We number the vertices $\{1,\dots,n-2\}$ (from lowest leaf to highest). We then attach the contracted ends $p_j$, in order of their size, to the vertices with numbers in $S$ (see figure \ref{connect_fig_std_max} for an example).

\begin{figure}[ht]
 \centering
 \begin{tikzpicture}
   \draw (0,0) -- (4,0) ;
   \draw (0,0) -- (-0.5,0.3) node[left]{$x_5$};
   \draw (0,0) -- (-0.5,-0.3) node[left]{$x_4$};
   \draw (1,0) -- (0.7,-0.5) node[below]{$x_1$};
   \draw (2,0) -- (2,-0.5) node[below]{$x_6$};
    \draw (3,0) -- (3.3,-0.5) node[below]{$x_7$};
    \draw (4,0) -- (4.5,0.3) node[right]{$x_2$};
    \draw (4,0) -- (4.5,-0.3) node[right]{$x_3$};
   \draw (0,0) node[above = 3pt]{$p_0$};
    \draw (2,0) node[above = 3pt]{$p_1$};
   \draw (3,0) node[above = 3pt]{$p_2$};
   \fill[black] (0,0) circle (2pt);
   \fill[black] (2,0) circle (2pt);
   \fill[black] (3,0) circle (2pt);
 \end{tikzpicture}
 \caption{The sorted maximal cover in $\hw{2}(x)$ for $x = (1,2,3,-3,-5,1,1)$ on $S = \{1,3,4\}$}\label{connect_fig_std_max}
\end{figure}
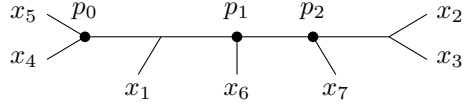

It is easy to see that all sorted maximal covers are connected in codimension one: assume that $(j-1) \notin S \ni j$ (i.e.\ there is a contracted end at vertex $j$ but none at vertex $(j-1)$). Then the sorted cover on $(S \wo \{j\}) \cup \{j-1\}$ shares a codimension one face with this cover, obtained by shrinking the edge between the two vertices $(j-1),j$ to length 0. In this manner we see that every sorted cover is connected to the sorted maximal cover on $S = \{1,\dots,n-2-k\}$.

Now we want to see that every maximal cell $\sigma$ is connected to the maximal cone of a sorted cover. The cell $\sigma$ corresponds to a trivalent curve, with some of the vertices marked with contracted ends $p_0,\dots,p_{n-3-k}$. We now add a further marking $q \in \R$ on an arbitrary vertex such that it is compatible with the edge directions. This gives us an element of $\hwt{k-1}(x,p)$. By induction, the corresponding cell is connected to a sorted cover on $S'$, with $\abs{S'} = n-3-k$. We can \enquote{lift} each intermediate step in the connecting path to a valid cover in $\hwt{k}(x,p)$ simply by forgetting the mark $q$. Thus we have connected $\sigma$ to a sorted maximal cover.
\end{proof}
\end{theorem}

\subsection{Irreducibility}\label{section_irreducible}

We now want to see when a Hurwitz cycle is irreducible. We just proved that it is connected in codimension one, so we can try to apply Proposition \ref{intro_prop_irred}. 
To see whether a Hurwitz cycle is locally irreducible, we will make use of our knowledge of the local structure of $\mk{N}$ (Corollary 6.18 in \cite{hatint}): if $\tau$ is a cone of the combinatorial subdivision of $\mk{N}$, corresponding to a curve $C$ with vertices $q_1,\dots,q_k$, then
$$\Star_{\mk{N}}(\tau) = \mk{\val(q_1)} \times \dots \times \mk{\val(q_k)}.$$

\begin{lemma}
 For any $x \in \curly{H}_n$ and pairwise different $p_j$, the cycle $\hwt{k}(x,p)$ is locally at each codimension one face weakly irreducible.
\end{lemma}
\begin{proof}
  Let $\tau$ be a codimension one cell of $\hwt{k}(x,p)$ and $C_\tau$ the corresponding combinatorial type. Since we chose the $p_j$ to be pairwise different, $C_\tau$ has exactly one vertex $v$ adjacent to four bounded edges or non-contracted leaves. Depending on whether a contracted end is also attached, the vertex is either four- or five-valent, corresponding to an $\curly{M}_4$- or $\curly{M}_5$-coordinate.

Denote by $S := \Star_{\hwt{k}(x)}(\tau)$. First let us assume that no contracted end is attached to $v$. Then there are three maximal cones adjacent to $\tau$, corresponding to the three different possible resolutions of $v$. The projections of the normal vectors to the $\curly{M}_4$-coordinate of $v$ are (multiples of) the three rays of $\curly{M}_4$. In particular there is only one possible way to assign weights to these rays so that they add up to 0. Hence the rank of $\Omega_S$ is 1, showing that $S$ is a multiple of an irreducible cycle.

Now assume there is a contracted end $p$ at $v$ and four edges/non-contracted ends. Then there are six maximal cones adjacent to $\tau$: consider $v$ as a four-valent vertex with an additional point for the contracted end. Then we still have three possibilities to resolve $v$, but in each case we have two possibilities to place the additional point (see figure \ref{figure_irred_six_res}). Now label the four ends and $p$ with numbers $1,\dots,5$ and assume $p$ is labeled with $5$. Then the projections of the normal vectors are multiples of the vectors $v_{\{i,j\}} \in \curly{M}_5$ with $i,j \neq 5$. The set of these vectors has been studied in \cite{kmpsiclasses} and it is shown there that there is only one way to assign weights to these rays such that they add up to 0.

\begin{figure}[ht]
 \centering
 \begin{tikzpicture}
     \matrix[column sep = 5mm, row sep = 3mm]{
	& 
	\draw (0,0) -- (-0.5,0.5) node[left]{1};
	\draw (0,0) -- (-0.5,-0.5) node[left]{2};
	\draw (0,0) -- (0.5,-0.5) node[right]{3};
	\draw (0,0) -- (0.5,0.5) node[right]{4};
	\draw[dotted] (0,0) -- (0,0.5) node[above]{$p$}; 
	\fill[black] (0,0) circle (2pt);
	& 
\\
& \draw[->] (0,0) -- (0,-0.5);
\\
	\draw (-0.2,0) -- (0.2,0);
	\draw[dotted] (-0.2,0) -- (-0.2,0.5) node[above]{$p$};
	\draw (-0.2,0) -- (-0.7,0.5) node[left]{1};
	\draw (-0.2,0) -- (-0.7,-0.5) node[left]{2};
	\draw (0.2,0) -- (0.7,0.5) node[right]{3};
	\draw (0.2,0) -- (0.7,-0.5) node[right]{4}; 
	\fill[black] (-0.2,0) circle (2pt);
	&
\draw (-0.2,0) -- (0.2,0);
	\draw[dotted] (-0.2,0) -- (-0.2,0.5) node[above]{$p$};
	\draw (-0.2,0) -- (-0.7,0.5) node[left]{1};
	\draw (-0.2,0) -- (-0.7,-0.5) node[left]{2};
	\draw (0.2,0) -- (0.7,0.5) node[right]{3};
	\draw (0.2,0) -- (0.7,-0.5) node[right]{4}; 
	\fill[black] (-0.2,0) circle (2pt);
	&
\draw (-0.2,0) -- (0.2,0);
	\draw[dotted] (-0.2,0) -- (-0.2,0.5) node[above]{$p$};
	\draw (-0.2,0) -- (-0.7,0.5) node[left]{1};
	\draw (-0.2,0) -- (-0.7,-0.5) node[left]{2};
	\draw (0.2,0) -- (0.7,0.5) node[right]{3};
	\draw (0.2,0) -- (0.7,-0.5) node[right]{4}; 	
	\fill[black] (-0.2,0) circle (2pt);
\\
\draw (-0.2,0) -- (0.2,0);
	\draw[dotted] (0.2,0) -- (0.2,0.5) node[above]{$p$};
	\draw (-0.2,0) -- (-0.7,0.5) node[left]{1};
	\draw (-0.2,0) -- (-0.7,-0.5) node[left]{2};
	\draw (0.2,0) -- (0.7,0.5) node[right]{3};
	\draw (0.2,0) -- (0.7,-0.5) node[right]{4}; 
	\fill[black] (0.2,0) circle (2pt);
	&
\draw (-0.2,0) -- (0.2,0);
	\draw[dotted] (0.2,0) -- (0.2,0.5) node[above]{$p$};
	\draw (-0.2,0) -- (-0.7,0.5) node[left]{1};
	\draw (-0.2,0) -- (-0.7,-0.5) node[left]{2};
	\draw (0.2,0) -- (0.7,0.5) node[right]{3};
	\draw (0.2,0) -- (0.7,-0.5) node[right]{4}; 
	\fill[black] (0.2,0) circle (2pt);
	&
\draw (-0.2,0) -- (0.2,0);
	\draw[dotted] (0.2,0) -- (0.2,0.5) node[above]{$p$};
	\draw (-0.2,0) -- (-0.7,0.5) node[left]{1};
	\draw (-0.2,0) -- (-0.7,-0.5) node[left]{2};
	\draw (0.2,0) -- (0.7,0.5) node[right]{3};
	\draw (0.2,0) -- (0.7,-0.5) node[right]{4}; 	
	\fill[black] (0.2,0) circle (2pt);
\\
      };
 \end{tikzpicture}
\caption{The six possible resolutions of a four-valent vertex with a contracted end.}\label{figure_irred_six_res}
\end{figure}
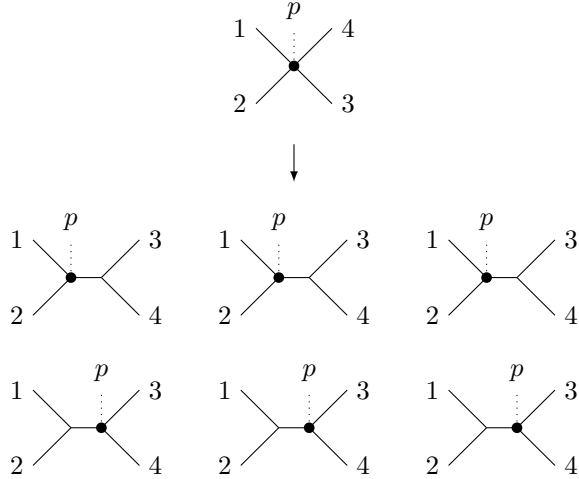
\end{proof}

\begin{corollary}\label{main_cor_irred}
  For any $x \in \curly{H}_n$ and any pairwise different $p_j$, $\hwt{k}(x,p)$ is weakly irreducible.
\end{corollary}

\begin{ex}\label{hurwitz_ex_irred}
We now want to see that this is the strongest possible statement (see also the subsequent \polymake\ example). 
\begin{itemize}
 \item \textbf{Non-generic points: } Let $n = 5, k = 1,  x = (1,1,1,1,-4)$. If we choose $p_0 = p_1 = 0$, then $\hwt{1}(x,p)$ is not irreducible: one can use \atint\ to compute that the rank of $\Omega_{\hwt{1}(x)}$ is 3. 
 \item \textbf{Strict irreducibility: } Let $x' = (1,1,1,1,1,-5), k' = 1$. For $(p_0,p_1,p_2) = (0,1,2)$, we obviously obtain a cycle with weight lattice $\Omega_{\hwt{1}(x',p)}$ of rank 1. However, the $\gcd$ of all weights in this cycle is 2, so it is not irreducible in the strict sense.
 \item \textbf{Unmarked cycles: } Again, choose $x = (1,1,1,1,-4), k' = 1$ and generic points $p_0 = 0,p_1 = 1$. Passing to $\hw{1}(x,p)$, the rank of $\Lambda_{\hw{1}}(x,p)$ is 18. One can also see that $\hw{1}(x,p)$ is not locally irreducible: it contains the two lines $\{ \frac{1}{2} v_{\{1,2\}} + \R_{\geq 0} v_{\{3,4\}}\}, \{\frac{1}{2} v_{\{3,4\}} + \R_{\geq 0} v_{\{1,2\}}\}$, which intersect transversely in the vertex $ \frac{1}{2} v_{\{1,2\}} + \frac{1}{2} v_{\{3,4\}}$.  Locally at this vertex, the curve is just the union of two lines, which is of course not irreducible. However, one can again use the computer to see that there are also vertices of $\hwt{1}(x,p)$ such that the map induced locally by $\ft$ is injective, but such that the image of the local variety at that vertex is not irreducible. 
\end{itemize}

 \begin{algorithm}[ht]
  \caption*{\textbf{polymake example: computing Hurwitz cycles}.\\ 
We compute the Hurwitz cycles from Example \ref{hurwitz_ex_irred}. First, we compute the cycle $\hwt{1}((1,1,1,1,-4),p)$ for $p_0 = p_1 = 0$ (If no points are given, they are set to 0). A basis for its weight space is given as row vectors of a matrix. We then compute $\hwt{1}( (1,1,1,1,1,-5),q)$ for generic points $q = (0,1,2)$ (the first point is always zero in \atint) and display its weight space dimension. Finally we compute $\hw{1}(1,1,1,1,-4)$ for generic points $(0,1)$ and the dimension of its weight space.}
  \flushleft
  
  \texttt{
\setlength{\tabcolsep}{2pt}
  \begin{tabular}{l l}
   \atcommand & \$h1 = hurwitz\_marked\_cycle(1,(new Vector<Int>(1,1,1,1,-4)));\\
   \atcommand & print \$h1-\>WEIGHT\_SPACE-\>rows();\\
   \multicolumn{1}{l}{3}\\
   \atcommand & \$h2 = hurwitz\_marked\_cycle(1,(new Vector<Int>(1,1,1,1,1,-5)),\\
  &(new Vector<Rational>(1,2)));\\
  \atcommand & print \$h2->WEIGHT\_SPACE->rows();\\
  \multicolumn{1}{l}{1}\\
  \atcommand & print gcd(\$h2->TROPICAL\_WEIGHTS);\\
  \multicolumn{1}{l}{2}\\
  \atcommand & \$h3 = hurwitz\_cycle(1,(new Vector<Int>(1,1,1,1,-4)),\\
  &(new Vector<Rational>([1])));\\
  \atcommand & print \$h3->WEIGHT\_SPACE->rows();\\
  \multicolumn{1}{l}{18}\\
  \end{tabular}}
\end{algorithm}

\end{ex}
\begin{remark} 
So far, we haven't found a single example of an irreducible Hurwitz cycle $\hw{k}(x,p)$. If we pick $p = 0$, it is actually obvious that the cycle must be reducible: for any $i = 1,\dots,n$ it contains the Psi class product $\psi_i^{(n-3-k)}$ as a non-trivial $k$-dimensional subcycle. In fact, finding a canonical decomposition, e.g.\ in terms of Psi class products, would be a large step towards finding a higher-dimensional ELSV formula. However, while possible decompositions can be found with \atint, the problem proves computationally infeasible in all but the smallest cases. 
\end{remark}

\subsection{Cutting out Hurwitz cycles}\label{section_hurwitz_cutting}
For intersection-theoretic purposes it is very tedious to have a representation of the cycle $\hw{k}(x)$  only as a push-forward. We would like to find rational functions that successively cut out (recession fans of) Hurwitz cycles directly in the moduli space $\mk{n}$. It turns out that there is a very intuitive rational function cutting out the codimension one Hurwitz cycle $\hw{n-4}(x)$ in $\mk{n}$. Alas, this seems to be the strongest possible statement already that we can make in this generality. For $n \geq 7$ we can find examples where there is \emph{no rational function at all} that cuts out $\hw{n-5}$ from $\hw{n-4}(x)$. It remains to be seen whether there might be other rational functions or piecewise polynomials cutting out lower-dimensional Hurwitz cycles from $\mk{n}$.

Throughout this section we assume $p_i = 0$ for all $i$, i.e.\ $\hw{k}$ is a fan in $\mk{n}$.

\subsubsection{Push-forwards of rational functions}

We already know that $\hwt{n-4}(x)$ can by definition be cut out from $$\ev_0^*(0) \cdot \Psi_0 \cdot \Psi_1 \cdot \mk{2}(\R,x) =: \curly{M}_x$$ by the rational function $\ev_1^*(0)$ (Note that there is an obvious isomorphism $\curly{M}_x \cong \psi_{n+1} \cdot \psi_{n+2} \cdot \mk{n+2}$). The forgetful map $\ft:\mk{2}(\R,x) \to \mk{n}$ now induces a (surjective) morphism of equidimensional tropical varieties (by abuse of notation we also denote it by $\ft$) $$\ft: \curly{M}_x \to \mk{n},$$
which is injective on each cone of $\curly{M}_x$.
We will see that under these conditions, we can actually define the \emph{push-forward} of a rational function. Note that we call a tropical variety $X$ \emph{smooth}, if it is locally at each point isomorphic to a matroidal fan (modulo some linear space). We will not go into the details of matroids and matroidal fans, which can for example be found in \cites{akbergman, fsmatroidpoly, stropicallinear}.

\begin{defn}
 Let $X,Y$ be $d$-dimensional tropical cycles and assume $Y$ is smooth. Let $x \in X$. If $f: X \to Y$ is a morphism, we denote by $f_x$ the induced local map
 $$f_x: \Star_X(x) \to \Star_Y(f(x)) =: V_x.$$
 We define the \emph{mapping multiplicity} of $x$ to be 
 $$m_x := f_x^*(f(x)).$$
 Note that, since $V_x$ is a smooth fan, any two points in it are rationally equivalent by \cite{frdiagonal}*{Theorem 9.5}, so $\deg f_x^*(\cdot)$ is constant on $V_x$. In particular, to compute $m_x$, we can replace $f(x)$ by any point $y$ in a sufficiently small neighborhood.
 
 Now let $g: X \to \R$ be a rational function. We define the \emph{push-forward} of $g$ under $f$ to be the function
 $$f_*g: Y \to \R, y \mapsto \sum_{x: f(x) = y} m_x g(x).$$
\end{defn}

\begin{prop}\label{hurwitz_prop_pushrational}
 Under the assumptions above, $f_*g$ is a rational function on $Y$.
 \begin{proof}
  We can assume without restriction that $X$ and $Y$ have been refined in such a manner that $f$ maps cells of $X$ to cells of $Y$ and $g$ is affine linear on each cell of $X$. Let us first see that $f_*g$ is well-defined:
  
  Let $y \in Y$ and denote by $\tau$ the minimal cell containing it. We want to see that $y$ has only finitely many preimages $x \in X$ with $m_x \neq 0$. Assume there is a cell $\rho$ in $X$ such that $f(\rho) = \tau$, but $\dim(\rho) > \dim(\tau)$, so $f_{\mid \rho}$ is not injective. In particular, all maximal cells $\xi > \rho$ map to a cell of dimension strictly less than $d$. Now let $x \in \relint(\rho)$ with $f(x) = y$. If we pick a point $q \in V_y$ that lies in a maximal cone adjacent to $\tau$, it has no preimage under $f_x$: all maximal cones in $\Star_X(x)$ are mapped to a lower-dimensional cone. It follows that $m_x = 0$.
  
  We now have to show that $f_*g$ is continuous. Let $\sigma$ be a maximal cell of $Y$. Denote by
  $$C_\sigma = \{\xi \in X^{(d)}, f(\xi) = \sigma\}.$$
  Then for each $y \in \relint(\sigma)$ we have
  $$f_*g(y) = \sum_{\xi \in C_\sigma} \omega_X(\xi) \textnormal{ind}(\xi) g(f_{\mid \xi}^{-1}({y})),$$
  where $\textnormal{ind}(\xi) := \abs{\Lambda_\sigma / f(\Lambda_\xi)}$ is the index of $f$ on $\xi$. Since $f_{\mid \xi}$ is a homeomorphism, this is just a sum of continuous maps, so $(f_*g)_{\mid \relint(\sigma)}$ is continuous.
  
  Assume $\tau$ is a cell of $Y$ of dimension strictly less than $d$ and contained in some maximal cell $\sigma$. Let $s:[0,1] \to \sigma$ be a continuous path with:
  \begin{itemize}
   \item $s([0,1)) \subseteq \relint(\sigma)$
   \item $s(1) \in \relint(\tau)$
  \end{itemize}
 We write $y_t := s(t)$ for $t \in [0,1]$. Then we have to show that $\lim_{t \to 1} f_*g(y_t) = f_*g(y_1)$. If we denote by $s_\xi = (f_{\mid \xi}^{-1} \circ s)$ the unique lift of $s$ to any $\xi \in C_\sigma$, we have
 \begin{align*}
    \lim_{t \to 1} f_*g(y_t) &= \lim_{t \to 1} \sum_{\xi \in C_\sigma} \omega_X(\xi) \textnormal{ind}(\xi) g(s_\xi(t)) \\
    &= \sum_{\xi \in C_\sigma} \omega_X(\xi) \textnormal{ind}(\xi) \lim_{t \to 1} g(s_\xi(t)) \\
    &= \sum_{\xi \in C_\sigma} \omega_X(\xi) \textnormal{ind}(\xi) g(\underbrace{\lim_{t \to 1} s_\xi(t)}_{=: x_\xi}),
 \end{align*}
where the last equality is due to the continuity of $g$. Note that $x_\xi$ lies in the unique face $\rho_\xi < \xi$ such that $f(\rho_\xi) = \tau$. 

Conversely, let $\rho$ be any cell of $X$ with $\dim(\rho) = \dim(\tau)$ and $f(\rho) = \tau$. Assume $\rho$ has no adjacent maximal cell mapping to $\sigma$. Then, if we let $x := f_{\mid \rho}^{-1}(y)$, we must again have $m_x = 0$. We define $$C_\tau := \{\rho \in X^{(\dim \tau)}; f(\rho) = \tau \textnormal{ and there exists } \xi > \rho \textnormal{ with } f(\xi) = \sigma\}.$$ Then we have 
\begin{align}
 \notag f_*g(y_1) &= \sum_{ \substack{\rho \in X^{(\dim \tau)}\\f(\rho) = \tau} } m_{f_{\mid \rho}^{-1}(y_1)} g(f_{\mid \rho}^{-1}(y_1))  \\
 &= \sum_{ \rho \in C_\tau } m_{f_{\mid \rho}^{-1}(y_1)} g(f_{\mid \rho}^{-1}(y_1))\label{hurwitz_eq_push} 
\end{align}
If $x_\rho := f_{\mid \rho}^{-1}(y_1)$, then for small $\epsilon$ we have
 $$m_{f_{\mid \rho}^{-1}(y_1)} = \deg f_{x_\rho}^* y_{1-\epsilon} = \sum_{\substack{\xi > \rho\\ f(\xi) = \sigma}} \omega_X(\xi) \textnormal{ind}(\xi).$$ 
 If we plug this into (\ref{hurwitz_eq_push}), we see that each $\xi \in C_\sigma$ occurs exactly once (since $\xi$ cannot have two faces $\rho$ mapping to $\tau$ due to injectivity), so finally we have $\lim_{t \to 1} f_*g(y_t) = f_*g(y_1)$.
 \end{proof}
\end{prop}

\begin{prop}\label{hurwitz_prop_pushequal}
 Let $f:X \to Y$ be a morphism of $d$-dimensional tropical cycles. Assume $Y$ is smooth and $f$ is injective on each cell of $X$. Then
 $$f_*g \cdot Y = f_*(g\cdot X).$$
 \begin{proof}
  By studying this identity locally and dividing out lineality spaces we can assume that:
  \begin{itemize}
   \item $Y$ is a smooth one-dimensional tropical fan.
   \item $X = \coprod_{i=1}^r X_i$ is a disjoint union of one-dimensional tropical fan cycles.
   \item $f_{\mid X_i}: X_i \to Y$ is a linear map.
   \item $g$ is affine linear on each ray of $X_i$.
  \end{itemize}
We write $Z := f_*g \cdot Y$ and $Z' := f_*(g \cdot X)$. We have to show that $\omega_Z(0) = \omega_{Z'}(0)$. We know that
$$ \omega_{Z'}(0) = \sum_{i=1}^r \omega_{g \cdot X_i}(0) = \sum_{i=1}^r \sum_{\rho \in X_i^{(1)}} \omega_{X_i}(\rho) g(u_\rho),$$
where $u_\rho$ is the integer primitive generator of $\rho$. On the other hand we have
\begin{align*}
 \omega_Z(0) &= \sum_{\sigma \in Y^{(1)}} f_*g(u_\sigma) \\
	      &= \sum_{\sigma \in Y^{(1)}} \sum_{i=1}^r \sum_{\substack{\rho \in X_i^{(1)}\\ f(\rho) = \sigma}} \omega_{X_i}(\rho) \textnormal{ind}(\rho) g\left(\frac{u_\rho}{\textnormal{ind}(\rho)}\right).
\end{align*}
Obviously each ray $\rho$ can occur at most once in this sum and by assumption it occurs at least once. Hence we see that $\omega_Z(0) = \omega_{Z'}(0)$.
 \end{proof}
\end{prop}

\begin{ex}
 Note that the assumption that $f$ is injective on each cone is necessary. Consider the morphism depicted in Figure \ref{hurwitz_fig_badpush}: in this case we get that $f_*(g \cdot X) = 4$ and $f_*g\cdot Y = 2$.
 
 \begin{figure}[ht]
  \centering
  \begin{tikzpicture}
   \matrix[column sep = 5mm, row sep = 5mm]{
    \draw (0,0) -- (1,0) node[right]{ $(g' = 1)$};
    \draw (0,0) -- (0,1) node[above]{$(g' = 1)$};
    \draw (0,0) -- (-1,0) node[left]{$(g' = 1)$};
    \draw (0,0) -- (0,-1) node[below]{$(g' = 1)$};
    \fill[black] (0,0) circle(2pt);
    \draw (-3,0) node{$X :=$}; \\
    \draw[->] (0,1) -- (0,0);\\
    \draw (-1,0) node[left]{$( (f_*g)' = 1)$} -- (1,0) node[right]{$((f_*g)' = 1)$};
    \fill[black] (0,0) circle(2pt);
    \draw (-3.5,0) node{$Y :=$};\\
   };
  \end{tikzpicture}
    \caption{A morphism where the push-forward of a function does not give the same divisor as the push-forward of the divisor of this function. All weights are 1 and the function slopes of $g$ and $f_*g$ are given in brackets.}\label{hurwitz_fig_badpush}
 \end{figure}
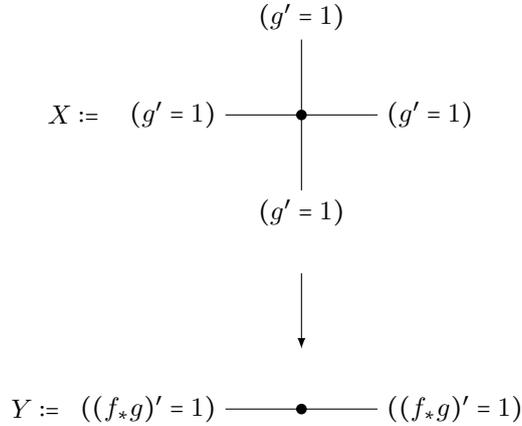
\end{ex}

\subsubsection{Cutting out the codimension one cycle}

By definition we have $$\hw{n-4}(x) = \ft_*(\hwt{n-4}(x)) = \ft_*(\ev_1^*(0) \cdot \curly{M}_x)$$
and we already discussed that $\ft: \curly{M}_x \to \mk{n}$ is a morphism of $(n-3)$-dimensional tropical varieties which is injective on each cone of $\curly{M}_x$. Since $\mk{n}$ is smooth, we immediately obtain the following result:

\begin{corollary}
 The codimension one Hurwitz cycle can be cut out as
 $$\hw{n-4} = (\ft_*(\ev_1^*(0))) \cdot \mk{n}.$$
\end{corollary}

 We now want to describe the rational function $(\ft_*(\ev_{1}^*(0)))$ in more intuitive and geometric terms: 

 \begin{lemma}\label{hurwitz_lemma_adddist}
 Let $C$ be any curve in $\mk{n}$. Given $x \in \curly{H}_n$ this defines a cover of $\R$ up to translation. Pick any such cover $h: C \to \R$. Let $v_1,\dots,v_r$ be the vertices of $C$. Then
  $$(\ft_*(\ev_1^*(0)))(C) = \sum_{i \neq j} (\val(v_i)-2)(\val(v_j)-2)\abs{h(v_i) - h(v_j)}.$$
\begin{proof}
 It suffices to show this for curves in maximal cones. Since $(\ft_*(\ev_1^*(0)))$ is continuous by Proposition \ref{hurwitz_prop_pushrational}, the claim follows for all other cones.
 
 So let $C$ be an $n$-marked trivalent curve with vertices $v_1,\dots,v_{n-2}$. We obtain all preimages in $\curly{M}_x$ by going over all possible choices of vertices $v_i,v_j$ and attaching the additional leaves $l_0$ to $v_i$ and $l_1$ to $v_j$. We denote the corresponding $n+2$-marked curve by $C(i,j)$. Note that $\ev_1$ maps $C(i,j)$ to the image of $l_1$ under the cover obtained by fixing the image of $l_0$ to be 0. We immediately see the following:
 \begin{itemize}
  \item $\ev_1(C(i,i)) = 0$.
  \item $\ev_1(C(i,j)) = - \ev_1(C(j,i))$.
  \item $\abs{\ev_1(C(i,j))} = \abs{h(v_i) - h(v_j)}$
 \end{itemize}
Since $\ev_1^*(0)(x) = \max\{0, \ev_1(x)\}$ and the forgetful map has index 1, the claim follows.
\end{proof}
\end{lemma}

\begin{figure}
 \centering
 \begin{tikzpicture}
 \matrix[column sep = 5mm,ampersand replacement=\&]{
    \draw[<-,shorten <= 2pt] (0,0) -- (1,0);
    \draw[->, shorten >= 2pt] (1,0) -- (2,0);
   \draw (0,0) -- (-0.5,0.3) node[left]{1};
   \draw (0,0) -- (-0.5,-0.3) node[left]{2};
   \draw (1,0) -- (0.7,0.5) node[above]{5};
%    \draw (1,0) -- (1.3,0.5) node[above]{4};
   \draw(2,0) -- (2.5,0.3) node[right]{4};
   \draw (2,0) -- (2.5,-0.3) node[right]{3};
    \draw (0.5,0) node[below]{\tiny $l = 1$};
    \draw (1.5,0) node[below]{ \tiny $l = 1/2$}; 
    \draw (1,0) node[below=10pt]{$C$};
    \fill[DarkGreen] (0,0) circle (2pt); 
    \fill[blue] (1,0) circle (2pt); 
    \fill[red] (2,0) circle (2pt);\&
    \draw[->] (0,0) -- (1,0); \&
%     \draw (0,0.5) -- (3,0.5) node[right]{\tiny 4};
    \draw (0,0.5) -- (2,0.5);
    \draw (2,0.5) -- (3,0.5)node[right]{\tiny 1};
    \draw (2,0.5) -- (3,0.2) node[right]{\tiny 2};
    \draw (0,0.5) -- (1,-0.1) ;
    \draw (1,-0.1) -- (3,-0.1) node[right]{\tiny 3};
    \draw (1,-0.1) -- (3,-0.3) node[right]{\tiny 4};
    \draw (0,0.5) -- (-1,0.5) node[left]{\tiny 5};
%     \draw[->] (1,-0.5) -- (1,-0.7);
    \draw (-1,-1) -- (3,-1) node[right]{\tiny $\R$};
    \fill[blue] (0,0.5) circle (2pt);
    \fill[DarkGreen] (2,0.5) circle (2pt);
    \fill[red] (1,-0.1) circle (2pt);
    \fill[blue, text = black] (0,-1) circle (2pt) node[below = 1.5pt]{\tiny $\alpha$} node[above]{\tiny $a_1$};
    \fill[red, text = black] (1,-1) circle (2pt) node[below]{\tiny $\alpha + 1$} node[above]{\tiny $a_2$};
    \fill[DarkGreen, text = black] (2,-1) circle (2pt) node[below]{\tiny $\alpha+2$} node[above]{\tiny $a_3$};
    \\
  };
\end{tikzpicture}
\caption{We compute $(\ft_*(\ev_1^*(0)))(C)$ for an example. We choose parameters $x = (1,1,1,1,-4)$ and $C = v_{\{1,2\}} + \frac{1}{2} v_{\{3,4\}}$. In this case Lemma \ref{hurwitz_lemma_adddist} tells us that the value of the function at $C$ is $\abs{a_2- a_1} + \abs{a_3 - a_1} + \abs{a_3 - a_2} = 1 + 2 + 1 = 4$.}
\end{figure}
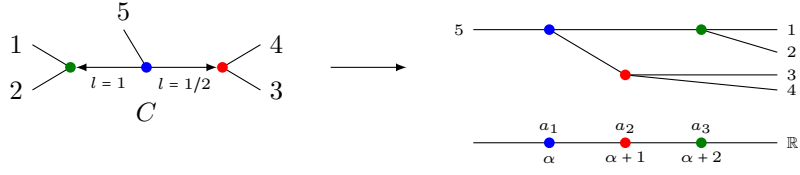

\subsection{Hurwitz cycles as linear combinations of boundary divisors}\label{section_hurwitz_boundary}

In \cite{bcmhurwitz}, the authors present several different representations of $\mathbb{H}_k(x)$. One is given in 

\begin{lemma}[{\cite{bcmhurwitz}*{Lemma 3.6}}]
 $$\mathbb{H}_k(x) = \sum_{\Gamma \in \mathcal{T}_{n-3-k}} \left(m(\Gamma) \varphi(\Gamma) \prod_{v \in \Gamma^{(0)}} (\textnormal{val}(v)-2)\Delta_\Gamma\right),$$
where $\Gamma$ runs over $\mathcal{T}_{n-3-k}$, the set of all combinatorial types of rational $n$-marked curves with $n-3-k$ bounded edges and $\Delta_\Gamma$ is the stratum of all covers with dual graph $\Gamma$. Furthermore, $m(\Gamma)$ is the number of total orderings on the vertices of $\Gamma$ compatible with edge directions and $\varphi(\Gamma)$ is the product over all edge weights.
\end{lemma}

There is an obvious, \enquote{naive} tropicalization of this: $\mathcal{T}_{n-3-k}$ corresponds to the codimension $k$ skeleton of $\mk{n}$. We will write $m(\tau) := m(\Gamma_\tau),x_\tau := \varphi(\Gamma_\tau)$ for any codimension $k$ cone $\tau$ and its corresponding combinatorial type $\Gamma_\tau$. The boundary stratum $\Delta_\Gamma$ we translate like this:

\begin{defn}
 
Let $(\curly{X},w)$ be a simplicial tropical fan. For a $d$-dimensional cone $\tau$ generated by rays $v_1,\dots,v_d$ we define rational functions $\varphi_{v_i}$ on $\curly{X}$ by fixing its value on all rays:
$$\varphi_{v_i}(r) = \begin{cases}
                      1, &\textnormal{ if } r = v_i\\
		      0, &\textnormal{ otherwise}
                     \end{cases}$$
for all $r \in \curly{X}^{(1)}$. We then write $\varphi_\tau := \varphi_{v_1} \cdot \dots \cdot \varphi_{v_d}$ for subsequently applying these $d$ functions. In the case of $\curly{X} = \mk{n}$ and $v_i = v_I$, we will also write $\varphi_I$ instead of $\varphi_{v_i}$.

As a shorthand notation we will write $\curly{C}_k$ for all dimension $k$ cells of $\mk{n}$ and $\curly{C}^k$ for all codimension $k$ cells (in its combinatorial subdivision).

Now we define the following divisor of a piecewise polynomial (see for example \cite{fcocycles} for a treaty of piecewise polynomials. For now it suffices if we define them as sums of products of rational functions):

$$D_k(x) := \sum_{\tau \in \curly{C}^k} m(\tau) \cdot x_\tau \cdot \left(\prod_{v \in \Gamma_\tau^{(0)}} (\val(v) -2)\right) \cdot \varphi_\tau \cdot \mk{n}, $$
where $\varphi_\tau = \prod_{v_I \in \tau^{(1)}} \varphi_I$ and the sum is to be understood as a sum of tropical cycles.
\end{defn}

We can now ask ourselves, what the relation between $D_k(x)$ and $\hw{k}(x)$ is. They are obviously not equal: $D_k(x)$ is a subfan of $\mk{n}$ (in its coarse subdivision), but even if we choose all $p_i$ to be equal to make $\hw{k}(x)$ a fan, it will still contain rays in the interior of higher-dimensional cones of $\mk{n}$.

This also rules out \emph{rational equivalence} (as defined in \cite{ahrrationalequivalence}): two cycles are equivalent, if and only if their recession fans are equal. 

But there is another, coarser equivalence on $\mk{n}$, that comes from toric geometry. As was shown in \cite{gmequations}, the classical $M_{0,n}$ can be embedded in the toric variety $X(\mk{n})$ and we have 
\begin{align*}
\textnormal{Cl}(X(\mk{n})) &\cong \textnormal{Pic}(\overline{M}_{0,n})\\
D_I &\mapsto \delta_I,
\end{align*}
where $D_I$ is the divisor associated to the ray $v_I$ and $\delta_I$ is the boundary stratum of curves consisting of two components, each containing the marked points in $I$ and $I^c$ respectively. By \cite{fstoricintersection}, $D_I$ corresponds to some tropical cycle of codimension one in $\mk{n}$ and \cite{rdiss}*{Corollary 1.2.19} shows that this is precisely $\varphi_I \cdot \mk{n}$. Hence the following is a direct translation of numerical equivalence in $\overline{M}_{0,n}$.

\begin{defn}
 Two $k$-dimensional cycles $C,D \subseteq \mk{n}$ are \emph{numerically equivalent}, if for all $k$-dimensional cones $\rho \in \mathcal{C}_k$ we have
$$\varphi_\rho \cdot C = \varphi_\rho \cdot D \in \Z.$$
\end{defn}

\begin{theorem}
 $\hw{k}(x)$ is numerically equivalent to $D_k(x)$.
\begin{proof}
  Note that for a generic choice of $p_i$, the cycle $\hw{k}(x)$ does not intersect any cones of codimension larger than $k$ and intersects all codimension $k$ cones transversely. For the proof we will need the following result from \cite{bcmhurwitz}*{Proposition 5.4}, describing the intersection multiplicity of $\hw{k}(x)$ with a codimension $k$-cell $\tau$: 
$$\tau \cdot \hw{k}(x) = m(\tau) \cdot x_\tau \cdot \prod_{v \in C_\tau^{(0)}} (\val(v) -2). $$
This implies that for any $\rho \in \curly{C}_k$ we have
\begin{align*}
 \varphi_\rho \cdot \hw{k}(x) &= \sum_{\tau \in \curly{C}^k} (\tau \cdot \hw{k}(x)) \cdot \omega_{\varphi_\rho \cdot \mk{n}}(\tau)\\
			&= \sum_{\tau \in \curly{C}^k} m(\tau) \cdot x_\tau \cdot \prod_{v \in C_\tau^{(0)}} (\val(v) -2) \cdot \omega_{\varphi_\rho \cdot \mk{n}}(\tau)\\
			&= \sum_{\tau \in \curly{C}^k} m(\tau) \cdot x_\tau \cdot \prod_{v \in C_\tau^{(0)}} (\val(v) -2) \cdot (\varphi_\tau \cdot \varphi_\rho \cdot \mk{n}) \\
			&= \varphi_\rho \cdot D_k(x),
\end{align*}
where $\omega_{\varphi_\rho \cdot \mk{n}}(\tau) = \varphi_\tau \cdot \varphi_\rho \cdot \mk{n}$ by \cite{fcocycles}*{Lemma 4.7}.

\end{proof}
\end{theorem}

\bibliographystyle{plain}
\bibliography{bibliography.bib}

\begin{bibdiv}
\begin{biblist}

\bib{ahrrationalequivalence}{article}{
      author={Allermann, Lars},
      author={Hampe, Simon},
      author={Rau, Johannes},
       title={On rational equivalence in tropical geometry},
        date={2014},
      eprint={1408.1537},
        note={preprint},
}

\bib{akbergman}{article}{
      author={Ardila, Federico},
      author={Klivans, Carly},
       title={The {B}ergman complex of a matroid and phylogenetic trees},
        date={2006},
     journal={J.\ Comb.\ Theory, Ser.\ B},
      volume={96},
       pages={38\ndash 49},
      eprint={math/0311370v2},
}

\bib{bbmopenhurwitz}{article}{
      author={Bertrand, Beno{\^{\i}}t},
      author={Brugall{\'e}, Erwan},
      author={Mikhalkin, Grigory},
       title={Tropical open {H}urwitz numbers},
        date={2011},
        ISSN={0041-8994},
     journal={Rend. Semin. Mat. Univ. Padova},
      volume={125},
       pages={157\ndash 171},
      eprint={1005.4628},
}

\bib{bcmhurwitz}{article}{
      author={Bertram, Aaron},
      author={Cavalieri, Renzo},
      author={Markwig, Hannah},
       title={Polynomiality, wall crossings and tropical geometry of rational
  double hurwitz cycles},
        date={2013},
     journal={JCTA},
      volume={120},
       pages={1604\ndash 1631},
}

\bib{beclassificationbranched}{article}{
      author={Berstein, Israel},
      author={Edmonds, Allan~L.},
       title={On the classification of generic branched coverings of surfaces},
        date={1984},
     journal={Illinois J. Math.},
      volume={28},
      number={1},
       pages={64\ndash 82},
}

\bib{bjsstcomputingtropical}{article}{
      author={Bogart, T.},
      author={Jensen, A.~N.},
      author={Speyer, D.},
      author={Sturmfels, B.},
      author={Thomas, R.~R.},
       title={Computing tropical varieties},
        date={2007},
        ISSN={0747-7171},
     journal={J. Symbolic Comput.},
      volume={42},
      number={1-2},
       pages={54\ndash 73},
         url={http://arxiv.org/abs/math/0507563},
}

\bib{ctheorieriemann}{article}{
      author={Clebsch, A.},
       title={Zur {T}heorie der {R}iemann'schen {F}l\"ache},
        date={1873},
     journal={Math. Ann.},
      volume={6},
      number={2},
       pages={216\ndash 230},
}

\bib{cjmhurwitz}{article}{
      author={Cavalieri, Renzo},
      author={Johnson, Paul},
      author={Markwig, Hannah},
       title={Tropical {H}urwitz numbers},
        date={2010},
        ISSN={0925-9899},
     journal={J. Algebraic Combin.},
      volume={32},
      number={2},
       pages={241\ndash 265},
      eprint={0804.0579},
         url={http://dx.doi.org/10.1007/s10801-009-0213-0},
}

\bib{cmrtropicalcompactification}{article}{
      author={Cavalieri, Renzo},
      author={Markwig, Hannah},
      author={Ranganathan, Dhruv},
       title={Tropical compactification and the gromov--witten theory of
  $\mathbb{P}^1$},
        date={2014},
      eprint={1410.2837},
        note={preprint},
}

\bib{cpconnectivity}{article}{
      author={Cartwright, Dustin},
      author={Payne, Sam},
       title={Connectivity of tropicalizations},
        date={2012},
     journal={Math. Res. Lett.},
      volume={19},
      number={5},
       pages={1089\ndash 1095},
      eprint={1204.6589},
}

\bib{elsvformula}{article}{
      author={Ekedahl, Torsten},
      author={Lando, Sergei},
      author={Shapiro, Michael},
      author={Vainshtein, Alek},
       title={Hurwitz numbers and intersections on moduli spaces of curves},
        date={2001},
        ISSN={0020-9910},
     journal={Inventiones mathematicae},
      volume={146},
      number={2},
       pages={297\ndash 327},
         url={http://dx.doi.org/10.1007/s002220100164},
}

\bib{fcocycles}{article}{
      author={Francois, Georges},
       title={Cocycles on tropical varieties via piecewise polynomials},
        date={2013},
     journal={Proceedings of the American Mathematical Society},
      volume={141},
       pages={481\ndash 497},
         url={http://arxiv.org/abs/1102.4783},
}

\bib{frdiagonal}{article}{
      author={Francois, Georges},
      author={Rau, Johannes},
       title={The diagonal of tropical matroid varieties and cycle
  intersections},
        date={2013},
     journal={Collectanea Mathematica},
      volume={64},
       pages={185\ndash 210},
         url={http://arxiv.org/abs/1012.3260},
}

\bib{fsmatroidpoly}{article}{
      author={Feichtner, Eva~Maria},
      author={Sturmfels, Bernd},
       title={Matroid polytopes, nested sets and {B}ergman fans},
        date={2005},
     journal={Portugaliae Mathematica. Nova S{\'e}rie},
      volume={62},
       pages={437\ndash 468},
         url={https://eudml.org/doc/52519},
}

\bib{fstoricintersection}{article}{
      author={{Fulton}, William},
      author={{Sturmfels}, Bernd},
       title={{Intersection theory on toric varieties}},
    language={English},
        date={1997},
        ISSN={0040-9383},
     journal={{Topology}},
      volume={36},
      number={2},
       pages={335\ndash 353},
}

\bib{ghsnoteonhurwitz}{article}{
      author={Graber, Tom},
      author={Harris, Joe},
      author={Starr, Jason},
       title={A note on hurwitz schemes of covers of a positive genus curve},
        date={2002},
      eprint={math/0205056},
        note={preprint},
}

\bib{gjpolymake}{incollection}{
      author={Gawrilow, Ewgenij},
      author={Joswig, Michael},
       title={polymake: a framework for analyzing convex polytopes},
        date={2000},
   booktitle={Polytopes---combinatorics and computation ({O}berwolfach, 1997)},
      series={DMV Sem.},
      volume={29},
   publisher={Birkh\"auser, Basel},
       pages={43\ndash 73},
}

\bib{gjvtowardsgeometry}{article}{
      author={Goulden, Ian},
      author={Jackson, David},
      author={Vakil, Ravi},
       title={Towards the geometry of double {H}urwitz numbers},
        date={2005},
     journal={Advances in Mathematics},
      volume={198},
      number={1},
       pages={43 \ndash  92},
      eprint={math/0309440},
}

\bib{gkmmoduli}{article}{
      author={Gathmann, Andreas},
      author={Kerber, Michael},
      author={Markwig, Hannah},
       title={Tropical fans and the moduli spaces of tropical curves},
        date={2009},
     journal={Compositio Mathematica},
      volume={145},
       pages={173\ndash 195},
      eprint={0708.2268},
         url={http://arxiv.org/pdf/0708.2268v2},
}

\bib{gmequations}{article}{
      author={Gibney, Angela},
      author={Maclagan, Diane},
       title={Equations for {C}how and {H}ilbert quotients},
        date={2010},
     journal={Algebra \& Number Theory},
      volume={7},
       pages={855\ndash 885},
      eprint={0707.1801},
         url={http://arxiv.org/pdf/0707.1801v2},
}

\bib{gvrelativevirtual}{article}{
      author={Graber, Tom},
      author={Vakil, Ravi},
       title={Relative virtual localization and vanishing of tautological
  classes on moduli spaces of curves},
        date={2005},
        ISSN={0012-7094},
     journal={Duke Math. J.},
      volume={130},
      number={1},
       pages={1\ndash 37},
      eprint={math/0309227},
         url={http://dx.doi.org/10.1215/S0012-7094-05-13011-3},
}

\bib{hatint}{article}{
      author={Hampe, Simon},
       title={a-tint: A polymake extension for algorithmic tropical
  intersection theory},
        date={2014},
     journal={European Journal of Combinatorics},
      volume={36},
       pages={579\ndash 607},
      eprint={1208.4248},
         url={http://arxiv.org/pdf/1208.4248v2},
}

\bib{hriemannsche}{article}{
      author={Hurwitz, Adolf},
       title={{\"U}ber {R}iemann'sche {F}l\"achen mit gegebenen
  {V}erzweigungspunkten},
        date={1891},
     journal={Math. Ann.},
      volume={39},
       pages={1\ndash 61},
}

\bib{jhurwitzoverview}{article}{
      author={Johnson, Paul},
       title={Hurwitz numbers, ribbon graphs and tropicalizations},
        date={2012},
     journal={Contemp. Math.},
      volume={580},
       pages={55\ndash 72},
      eprint={1303.1543},
}

\bib{kirreducibility}{article}{
      author={Kanev, Vassil},
       title={Irreducible components of {H}urwitz spaces parameterizing
  {G}alois coverings of curves of positive genus},
        date={2014},
     journal={Pure Appl. Math. Q.},
      volume={10},
      number={2},
       pages={193\ndash 222},
}

\bib{kintersectiontheory}{article}{
      author={Kontsevich, Maxim},
       title={Intersection theory on the moduli space of curves and the matrix
  {A}iry function},
        date={1992},
        ISSN={0010-3616},
     journal={Comm. Math. Phys.},
      volume={147},
      number={1},
       pages={1\ndash 23},
         url={http://projecteuclid.org/euclid.cmp/1104250524},
}

\bib{kmpsiclasses}{article}{
      author={Kerber, Michael},
      author={Markwig, Hannah},
       title={Intersecting psi-classes on tropical $\mathcal{M}_{0,n}$},
        date={2008},
     journal={International Mathematics Research Notices},
      eprint={0709.3953},
         url={http://dx.doi.org/10.1093/imrn/rnn130},
}

\bib{menumerative}{article}{
      author={Mikhalkin, Grigory},
       title={Enumerative tropical algebraic geometry in {$\mathbb{R}^2$}},
        date={2005},
        ISSN={0894-0347},
     journal={J. Amer. Math. Soc.},
      volume={18},
      number={2},
       pages={313\ndash 377},
      eprint={math/0312530},
         url={http://dx.doi.org/10.1090/S0894-0347-05-00477-7},
}

\bib{mmodulicurves}{inproceedings}{
      author={Mikhalkin, Grigory},
       title={Moduli spaces of rational tropical curves},
        date={2007},
   booktitle={Proceedings of {G}\"okova {G}eometry-{T}opology {C}onference
  2006},
   publisher={G\"okova Geometry/Topology Conference (GGT), G\"okova},
       pages={39\ndash 51},
}

\bib{mstropicalbook}{book}{
      author={Maclagan, Diane},
      author={Sturmfels, Bernd},
       title={Introduction to {T}ropical {G}eometry},
      series={Graduate Studies in Mathematics},
   publisher={American Mathematical Society, Providence, RI},
        date={2015},
      volume={161},
}

\bib{nstoricdegenerations}{article}{
      author={Nishinou, Takeo},
      author={Siebert, Bernd},
       title={Toric degenerations of toric varieties and tropical curves},
        date={2006},
     journal={Duke Math. J.},
      volume={135},
      number={1},
       pages={1\ndash 51},
      eprint={math/0409060},
}

\bib{ophurwitz}{article}{
      author={Okounkov, Andrei},
      author={Pandharipande, Rahul},
       title={Gromov-{W}itten theory, {H}urwitz numbers and matrix models,
  {I}},
        date={2001},
      eprint={math/0101147},
}

\bib{rdiss}{thesis}{
      author={Rau, Johannes},
       title={Tropical intersection theory and gravitational descendants},
        type={Ph.D. Thesis},
        date={2009},
         url={https://kluedo.ub.uni-kl.de/frontdoor/index/index/docId/2122},
}

\bib{salgebraische}{book}{
      author={Severi, Francesco},
       title={{V}orlesungen \"uber algebraische {G}eometrie},
   publisher={Teubner-Verlag},
        date={1921},
}

\bib{stropicallinear}{article}{
      author={Speyer, David},
       title={Tropical linear spaces},
        date={2008},
     journal={SIAM J.\ Discrete Math.},
      volume={22},
       pages={1527\ndash 1558},
      eprint={math/0410455},
}

\bib{sstropgrass}{article}{
      author={Speyer, David},
      author={Sturmfels, Bernd},
       title={The tropical {G}rassmannian},
        date={2004},
        ISSN={1615-715X},
     journal={Adv. Geom.},
      volume={4},
      number={3},
       pages={389\ndash 411},
      eprint={math/0304218},
         url={http://dx.doi.org/10.1515/advg.2004.023},
}

\bib{virreducibility}{article}{
      author={Vetro, Francesca},
       title={Irreducibility of {H}urwitz spaces of coverings with one special
  fiber and monodromy group a {W}eyl group of type {$D_d$}},
        date={2008},
     journal={Manuscripta Math.},
      volume={125},
      number={3},
       pages={353\ndash 368},
}

\end{biblist}
\end{bibdiv}

\end{document}